\documentclass[oneside]{amsart}
\usepackage{geometry}

\usepackage[utf8]{inputenc}
\usepackage{amsmath,amsfonts,amssymb,amsthm}
\usepackage[table]{xcolor}
\usepackage{graphicx,color,hyperref}

\usepackage[arrow, matrix, curve]{xy}

\usepackage[nameinlink]{cleveref} 
\crefname{prop}{Prop.}{Prop.}
\crefname{lem}{Lem.}{Lem.}
\crefname{thm}{Thm.}{Thm.}
\crefname{defn}{Def.}{Def.}
\crefname{cor}{Cor.}{Cor.}
\crefname{rem}{Rem.}{Rem.}
\crefname{expl}{Expl.}{Expl.}

\theoremstyle{plain}
\newtheorem{thm}{Theorem}[section]
\newtheorem{lem}[thm]{Lemma}
\newtheorem{prop}[thm]{Proposition}
\newtheorem{cor}[thm]{Corollary}

\theoremstyle{definition}
\newtheorem{defn}[thm]{Definition}

\theoremstyle{remark}
\newtheorem{rem}[thm]{Remark}
\newtheorem{expl}[thm]{Example}

\numberwithin{equation}{section}

\DeclareMathOperator\id{id}
\DeclareMathOperator\im{im}
\DeclareMathOperator\Hom{Hom}
\DeclareMathOperator\End{End}
\DeclareMathOperator\Tr{Tr}

\newcommand{\GL}{\mathsf{GL}}

\newcommand{\Biv}{\mathsf{Biv}}

\newcommand{\FF}{\mathbb{F}}
\newcommand{\CC}{\mathbb{C}}
\newcommand{\RR}{\mathbb{R}}
\newcommand{\ZZ}{\mathbb{Z}}

\newcommand{\U}{\mathcal{U}}
\renewcommand{\O}{\mathcal{O}}

\newcommand{\gfrak}{\mathfrak{g}}
\newcommand{\h}{\mathfrak{h}}
\newcommand{\gl}{\mathfrak{gl}}
\renewcommand{\sl}{\mathfrak{sl}}
\newcommand{\so}{{\mathfrak{so}}}

\newcommand{\grA}{{\overline A}}

\newcommand{\tiH}{{\tilde H}}
\newcommand{\tih}{{\tilde{h}}}

\newcommand{\tixi}{{\tilde\xi}}
\renewcommand\l\lambda
\newcommand{\tiLambda}{{\Lambda^+_\xi}}
\renewcommand\L\Lambda

\newcommand{\ieven}{^{\textnormal{even}}}

\newcommand\tforall{\qquad\text{for all }}

\newcommand{\hs}{\mathsf{hs}}
\newcommand{\Hcal}{\mathcal{H}}
\newcommand{\inv}{^{-1}}
\newcommand{\tensor}{\otimes}
\renewcommand{\o}{\otimes}
\newcommand{\eps}{\varepsilon}
\newcommand{\DeltaC}{{\Delta_C}}
\newcommand{\AC}{{A\tensor C}}
\newcommand{\tri}[1][\cdot,\cdot]{\langle #1 \rangle}

\newcommand{\xcite}[2][]{\cite[#1]{#2}}

\begin{document}

\title[Hopf--Hecke algebras, infinitesimal Cherednik algebras and Dirac cohomology]{Hopf--Hecke algebras, infinitesimal Cherednik algebras, \\ and Dirac cohomology}


\author{Johannes Flake}
\email{flake@art.rwth-aachen.de}
\address{Algebra and Representation Theory, RWTH Aachen University,
Pontdriesch 10--16,
52062 Aachen, Germany}

\author{Siddhartha Sahi}
\email{sahi@math.rutgers.edu}
\address{Department of Mathematics, Rutgers University,
Hill Center -- Busch Campus, 110 Frelinghuysen
Road, Piscataway, NJ 08854--8019, USA}

\begin{abstract} Hopf--Hecke algebras and Barbasch--Sahi algebras were defined by the first named author (2016) in order to provide a general framework for the study of Dirac cohomology. The aim of this paper is to explore new examples of these definitions and to contribute to their classification. Hopf--Hecke algebras are distinguished by an orthogonality condition and a PBW property. The PBW property for algebras such as the ones considered here has been of great interest in the literature and we extend this discussion by further results on the classification of such deformations and by a class of hitherto unexplored examples. We study infinitesimal Cherednik algebras of $\GL_n$ as defined by Etingof, Gan, and Ginzburg in [Transform.~Groups, 2005] as new examples of Hopf--Hecke algebras with a generalized Dirac cohomology. We show that they are in fact Barbasch--Sahi algebras, that is, a version of Vogan's conjecture analogous to the results of Huang and Pand\v{z}i{\'c} in [J.~Amer.~Math.~Soc., 2002] is available for them. We derive an explicit formula for the square of the Dirac operator and use it to study the finite-dimensional irreducible modules. We find that the Dirac cohomology of these modules is non-zero and that it, in fact, determines the modules uniquely.
\end{abstract}

\dedicatory{In fond memory of Bert Kostant: friend, philosopher, and guide.}

\maketitle

\setcounter{tocdepth}{2}
\tableofcontents

\section{Introduction}

The Dirac operator was introduced by Dirac in 1928 (\cite{Di}) in order to formulate a relativistic quantum mechanical equation for the electron. It has played an important role in many areas of physics and in differential geometry, especially in the Atiyah--Singer index theorem (\cite{AtS}) and related developments. An algebraic version of the Dirac operator was first employed by Parthasarathy (\cite{Pa}) to study the discrete series of a real reductive group $G$ (\cite{Vo}). It was subsequently applied to the study of unitary representations in general, perhaps with greatest effect to the classification of unitary highest weight representations (\cite{EHW, Ja}) and unitary representation with non-zero cohomology (\cite{VZ}).

In representation theory one studies the Harish--Chandra category \textsf{HC} of ``admissible'' modules for the pair $(\mathfrak{g},K)$ where $\mathfrak{g}$ is the complexified Lie algebra of $G$ and $K$ is a maximal compact subgroup. The irreducible objects in \textsf{HC} have been classified by Langlands and the key open problem is to identify the subset of unitarizable modules, which are in bijection with irreducible unitary representations of $G$. For any $M$ in \textsf{HC}, the algebraic Dirac operator $D$ acts on $M\otimes S$, where $S$ is a spin representation of $K$. For unitarizable $M$, this action is semisimple and, as shown by Parthasarathy, this leads to an inequality relating the actions of the Casimir operators of $G$ and $K$ acting on $M$ and $\ker(D)$, respectively.

Vogan suggested a far-reaching extension of these ideas to an arbitrary, not necessarily unitarizable, $M$ in \textsf{HC}. He proposed that one should study the Dirac cohomology
\[H^{D}(M):=\ker(D)/\ker(D)\cap
\operatorname{im}(D)
\]
and he conjectured that this space, if non-zero, should determine the full infinitesimal character of $M$ and not just the Casimir action. Vogan's conjecture was proved by Huang--Pand\v{z}i\'c \cite{HP} in the original setting, and these ideas have since been extended considerably to include various classes of Hecke algebras, see, e.g., \cite{BCT,Ci}. 

In \cite{Fl} the first author established a common generalization of these cases. The main results of \cite{Fl} are the definition of an extremely general class $\boldsymbol{H}$ of algebras, termed Hopf--Hecke algebras, for which one has a useful formulation of Dirac cohomology, as well as a precise characterization of a subclass $\boldsymbol{B}$ for which an analog of Vogan's conjecture is true. These latter algebras have been termed Barbasch--Sahi algebras by the first author to acknowledge unpublished contributions of D.~Barbasch and the second author in this direction.

Kostant (\cite{Ko-multiplets,Ko-bbw,Ko-hdm}) considered a cubic version of the Dirac operator for semisimple complex Lie algebras with a suitable reductive Lie subalgebra, generalizing the setting of Vogan and Huang--Pand\v{z}i\'c, where the subalgebra originates from a Cartan decomposition. He used this operator in proving a generalized Bott--Borel--Weil theorem and in studying the appearance of certain multiplets of representations. A version of Vogan's conjecture was proved for the cubic Dirac operator, generalizing the result of Huang--Pand\v{z}i\'c. Note that the cubic Dirac operator and its version of Vogan's conjecture are not a special case of the generalizations in \cite{Fl}.

The purpose of the present work is two-fold.
We take the first steps towards a classification of Hopf--Hecke algebras. We define a subclass $\boldsymbol{S}$ of $\boldsymbol{H}$ consisting of algebras that we refer to as standard, and which are obtained by an explicit construction. The relations between the various classes of algebras can be described by this diagram (see \Cref{rem-hs,rem-hs-diag}):
\newcommand\hsdiag{\[\begin{xy}
\xymatrix{
  \boldsymbol{B}~
    \ar@{^(->}[r]^{\neq}
  & \boldsymbol{H}
  \\
  \boldsymbol{B}\cap\boldsymbol{S}~
    \ar@{^(->}[r]^{\quad\neq}
    \ar@{^(->}[u]
  & \boldsymbol{S}
    \ar@{^(->}[u]
}
\end{xy}\]}
\hsdiag
We also describe a map $\hs$ from $\boldsymbol{H}$ to $\boldsymbol{S}$ which is idempotent, but ``lossy'', i.e.~not surjective onto $\boldsymbol{S}$, even though its image contains for instance all Hopf--Hecke algebras coming from finite groups. 

We also exhibit a new example of a Barbasch--Sahi algebra, thus showing that the class $\boldsymbol{B}$ is strictly larger than just the cases previously studied in the literature. This new example, actually family of examples, consists of infinitesimal Cherednik algebras $\mathcal{H}_{\xi}$, which are deformations of the enveloping algebra $U(\mathfrak{sl}_{n+1})$ parameterized by a polynomial $\xi=\xi(z)$. As members of $\boldsymbol{H}$ and $\boldsymbol{B}$, these deformations are not in the image of $\hs$ in general.

It is an interesting open problem to construct non-standard Hopf--Hecke algebras, or to prove that they do not exist. It is also of considerable interest to classify standard (and non-standard) Barbasch--Sahi algebras. We hope to return to these two problems in the near future.

\medskip

\textbf{Organization of this paper.} We recall from \cite{Fl} that Hopf--Hecke algebras are PBW deformations constructed from a cocommutative Hopf algebra $H$, an orthogonal module $V$ and a deformation map $\kappa$.

The PBW property of Hopf--Hecke algebras and their deformation maps $\kappa$ is explored in \Cref{sec-Hopf--Hecke-alg}. We recall the well-known relation between the PBW property and a general kind of Jacobi identity. We generalize methods developed for Drinfeld Hecke algebras (corresponding to the special case where $H$ is the group algebra of a finite group, see for instance \cite{RS}) to the general Hopf algebra case, and we adapt ideas from the setting of infinitesimal Hecke algebras (corresponding to the special case where $H$ is the universal enveloping algebra of a Lie algebra, see \cite{EGG}). In particular, we define an algebra filtration for any cocommutative Hopf algebra acting orthogonally on a module, which allows us to distinguish orders of the deformation map $\kappa$. We then use the coradical filtration to obtain more concrete information on $\kappa$ for Hopf algebras over $\CC$ (or, more generally, pointed Hopf algebras), and to define the ``standardization map'' $\hs$.

Some of the main results on the Dirac cohomology of Hopf--Hecke algebras and Barbasch--Sahi algebras from \cite{Fl} are recalled in \Cref{sec-Dirac}.

The Dirac cohomology of infinitesimal Cherednik algebras  of $\GL_n$ is studied in \Cref{sec-inf-Cherednik-alg}. We use an integral formula for the deformation map $\kappa$ in this special case to explicitly compute the square of the Dirac element, which will allow us to conclude that the infinitesimal Cherednik algebras of $\GL_n$ are, in fact, Barbasch--Sahi algebras. Finally, we use our formula for the square of the Dirac element to show that all finite-dimensional modules are determined by their Dirac cohomology.

\medskip

\textbf{Acknowledgments.} We would like to thank Apoorva Khare for his valuable feedback on an earlier version of this paper. We are also grateful to the referee for their useful comments, in particular the suggestion that we include an explicit version of the correspondence of central characters for infinitesimal Cherednik algebras (see \Cref{prop-explicit} below).

The research of S.~Sahi was partially supported by a Simons Foundation grant (509766) and an NSF grant (DMS-2001537).


\section{Hopf--Hecke algebras as PBW deformations} \label{sec-Hopf--Hecke-alg}

Let $\FF$ be a field of characteristic $0$. All vector spaces and tensor products are over $\FF$, all modules are finite-dimensional left modules.

We start with a brief review of basic Hopf algebra theory, for more information on this we refer to \cite{Mo}. When working with a coalgebra $C$ we will refer to its counit as $\eps:C\to\FF$ and to its coproduct as $\Delta:C\to C\tensor C$. When working with a Hopf algebra $H$, the same convention applies and the antipode is referred to as $S:H\to H$. For an element $c\in C$, we will use Sweedler's notation $c_{(1)}\tensor c_{(2)}$ for the coproduct $\Delta(c)\in C\tensor C$  which does not necessarily represent a pure tensor, but implies a summation over several pure tensors in general. The $i$-fold coproduct for $c$ is written as $c_{(1)}\tensor\dots\tensor c_{(i+1)}$ in $C^{\tensor (i+1)}$, where the notation is justified by the coassociativity. The Hopf algebra $H$ is an $H$-module itself (in fact, an $H$-module algebra) via the \emph{(left) adjoint action}
\[ h\cdot k := h_{(1)} k S(h_{(2)})
\qquad\text{for } h,k\in H
\ .
\]

\begin{defn} A coalgebra $C$ is called \emph{pointed} if every simple subcoalgebra is one-dimensional. An element $c\in C$ is called \emph{group-like} if $\Delta c=c\tensor c$ and $\eps(c)=1$, and the set of group-like elements is denoted by $G(C)$. If $H$ is a bialgebra, an element $h\in H$ is called \emph{primitive} if $\Delta h=1\tensor h+h\tensor 1$ and $\eps(h)=0$, and the set of primitives is denoted by $P(H)$.
\end{defn}

Basic Hopf algebra theory tells us that for every Hopf algebra $H$, $G(H)$ is a group with multiplication in $H$ and $P(H)$ is a Lie subalgebra of $H$ with the commutator.

\begin{defn} If $H$ is a Hopf algebra and $B$ is an $H$-module algebra, then the \emph{semidirect/smash product} $B\rtimes H$ is the algebra generated by $B$ and $H$ with the additional relation $hb=(h_{(1)} \cdot b) h_{(2)}$ for all $h\in H,b\in B$, that is, $B\rtimes H\simeq B\o H$ as a vector space and
\[ (b\o h) (b'\o h') = b (h_{(1)}\cdot b') \o h_{(2)} h' 
\tforall b,b'\in B,h,h'\in H
\ .
\]
We frequently identify $B$ with the subalgebra $B\o1$ and $H$ with the subalgebra $1\o H$ of $B\rtimes H$.
\end{defn}

By a well-known structure theorem, any cocommutative pointed Hopf algebra $H$ over a field $\FF$ of characteristic $0$ has the form $H=\U(P(H))\rtimes \FF[G(H)]$, where $\U(P(H))$ is the universal enveloping algebra of the Lie algebra of primitive elements and $\FF[G(H)]$ is the group algebra of the group of group-like elements in $H$. This applies, in particular, to any cocommutative Hopf algebra over $\FF=\CC$, because every simple cocommutative coalgebra over an algebraically closed field is one-dimensional.

\subsection{The PBW property and the Jacobi property} \label{subsec-Hopf--Hecke} We review the definition of Hopf--Hecke algebras (\xcite[Def.~3.1]{Fl}) and we will study their structure. To this end we fix a cocommutative Hopf algebra $H$ and a finite-dimensional $H$-module $V$. As $H$ is cocommutative, the tensor algebra $T(V)$ is an $H$-module algebra. The semidirect/smash product $T(V)\rtimes H$ is the algebra generated by $T(V)$ and $H$ and the relation
\[ hv = (h_{(1)} \cdot v) h_{(2)}
\tforall h\in H, v\in V
\ .
\]

\begin{defn} A bilinear form $\tri$ on $V$ is called \emph{$H$-invariant} if $\tri[h_{(1)}\cdot v,h_{(2)}\cdot w]=\eps(h)\tri[v,w]$, or equivalently if $\tri[h\cdot v,w]=\tri[v,Sh\cdot w]$, for all $h\in H$, $v,w\in V$ (\xcite[Lem.~2.3]{Fl}). $V$ is called an \emph{orthogonal} module if it admits a non-degenerate $H$-invariant symmetric bilinear form.
\end{defn}

In \xcite[Sec.~2]{Fl}, a \emph{pin cover} of $H$ with respect to $V$ is constructed for any pointed cocommutative Hopf algebra $H$ over $\FF$ with an orthogonal module $V$ (the special case relevant for this paper will be defined in \Cref{defn-pin-cover-Ug}). Note also that for such Hopf algebras, $V$ is orthogonal if and only if every group-like element acts as an orthogonal operator and every primitive element acts as a skew-symmetric operator.

\begin{defn}\label{defn-kappa} Let $\kappa:V\wedge V\to H$ be an $\FF$-linear map. We denote by $I_\kappa$ the two-sided ideal of $T(V)\rtimes H$ generated by elements of the form $vw-wv-\kappa(v\wedge w)$ for $v,w\in V$. The algebra
\begin{equation}
A=A_{H,V,\kappa}:=(T(V)\rtimes H)/I_\kappa
\end{equation}
is called a \emph{Hopf--Hecke algebra} if $V$ is an orthogonal module and if it satisfies the \emph{PBW property}, that is, if it is a \emph{flat deformation} of $S(V)\rtimes H$.
\end{defn}

In other words, let $\grA$ be the associated graded algebra of $A$ with respect to the filtration of the tensor factor $T(V)$. Now $A$ satisfies the PBW property if the natural surjection from $S(V)\rtimes H$ to $\grA$ is an isomorphism.

\begin{rem} We review \xcite[Rem.~3.2]{Fl}, since it will make the following structure theory more transparent: First, we note that the definition of a Hopf--Hecke algebra is closely related to that of continuous Hecke algebras in \cite{EGG}; if $G$ is a reductive algebraic group and $\gfrak$ its Lie algebra, then the Hopf algebra $H=\U(\gfrak)\rtimes\FF[G]$ can be viewed as a subalgebra of the algebra of algebraic distributions $\O(G)^*$ on $G$. If we replace $H$ with $\O(G)^*$ in the definition above and drop the orthogonality condition on $V$, we have the definition of continuous Hecke algebras in the sense of \cite{EGG}.

Second, we observe that a special case of our definition is the situation of $H$ being the group algebra of a finite group $G$. In this context, the algebras $A_{H,V,\kappa}$ have been studied in \cite{Dr, RS, SW}. If $\FF=\RR$, then every module $V$ is orthogonal, because any positive-definite symmetric bilinear form can be averaged to obtain an invariant positive-definite symmetric bilinear form.
\end{rem}

\begin{defn}\label{defn-Jacobi-property} We say that an $\FF$-linear map $\kappa:V\wedge V\to H$ (as in \Cref{defn-kappa}) is \emph{$H$-equivariant} if $\kappa(h\cdot r)=h\cdot \kappa(r)$ for all $h\in H, r\in V\wedge V$, and we say that $\kappa$ has the \emph{Jacobi property} if the following \emph{Jacobi identity} holds in $A_{H,V,\kappa}$ for all $x,y,z\in V$:
\[ [\kappa(x,y),z]+ [\kappa(y,z),x]+ [\kappa(z,x),y] = 0
\ .
\]
\end{defn}

The following fact is well-known (\cite{BG}, \cite[Thm.~2.4]{EGG}, \cite[Thm.~3.1]{WW}, \cite[Thm.~2.5]{Kh}):

\begin{prop} \label{prop-Jacobi-identity} $A_{H,V,\kappa}$ has the PBW property if and only if $\kappa$ is $H$-equivariant and $\kappa$ has the Jacobi property.
\end{prop}

\newcommand{\actson}{\triangleright}
In order to study the Jacobi property, we introduce some useful notation: for all $h\in H, v\in V$,
\begin{equation}
 h\actson v:=h \cdot v-\eps(h)v
\ .
\end{equation}
Note that the triangle ``$\actson$'' just denotes an $\FF$-linear action of $H$, not an algebra action of $H$ (in contrast to the dot ``$\cdot$'').

\begin{defn} For any $i\geq 0$, we define $K'_i\subset H$ to  be the subspace of those $h\in H$ satisfying
\begin{equation} \label{eq-symplectic-h}
(h_{(1)}\actson v_1) \wedge \dots
  \wedge (h_{(i+1)}\actson v_{i+1})  = 0
\end{equation}
for all $v_1,\dots,v_{i+1}\in V$. Let $K_i:=\Delta\inv(K'_i\tensor H)$.
\end{defn}

\begin{rem} Note that the left-hand side of \eqref{eq-symplectic-h} can be expanded. For instance, for $i=1$, the expansion reads
\[
 h \cdot (v_1\wedge v_2)
 - (h \cdot v_1)\wedge v_2 - v_1\wedge (h \cdot v_2)
 +\eps(h) v_1\wedge v_2
\ ,
\]
where the dot denotes the action of $H$ on $\Lambda(V)$, and similar expansions of the $\actson$-action in terms of the usual action of $H$ on the exterior algebra of $V$ exist for all $i\geq 1$.

Note also that $K'_0=\{h\in H: h \cdot v=\eps(h)v\}$, and that due to cocommutativity, $\Delta\inv(K'_i\tensor H)=\Delta\inv(K'_i\tensor K'_i)$ for all $i\geq 0$.
\end{rem}

\begin{lem} $(K_i)_{i\geq 0}$ is an algebra filtration of $H$.
\end{lem}

\begin{proof} First, we want to show $K_i\subset K_{i+1}$ for any $i\geq 0$. We consider $h\in K_i$ and we write $\Delta h=\sum_k r^k\tensor h^k$ with $(h^k)_k$ in $H$ and $(r^k)_k$ in $K'_i$. Then
\begin{align*}
 &(h_{(1)}\actson v_1) \wedge \dots \wedge (h_{(i+2)}\actson v_{i+2}) \tensor h_{(i+3)} \\
 &= \sum_k (r^k_{(1)}\actson v_1) \wedge \dots \wedge (r^k_{(i+1)}\actson v_{i+1}) \wedge (h^k_{(1)}\actson v_{i+2}) \tensor h^k_{(2)}
  = 0
\ ,
\end{align*}
for any $v_1,\dots,v_{i+2}\in V$, so $h\in K_{i+1}$, as desired.

To see that we obtain an algebra filtration, consider $i,j\geq 0$ and let $m:=i+j$. If $a,b\in H$, then
\[ (ab)\actson v = a \cdot (b\actson v)+\eps(b)(a\actson v)
\tforall v\in V
\ .
\]
Let us use the shorthand notations $S(a,b,v):=a \cdot (b\actson v)$ and $T(a,b,v):=\eps(b)(a\actson v)$. Then for all $v_1,\dots,v_{m+1}$ in $V$,
\begin{align*}
 &((ab)_{(1)}\actson v_1) \wedge\dots\wedge
 ((ab)_{(m+1)}\actson v_{m+1})
 \tensor (ab)_{(m+2)} \\
 &=((a_{(1)}b_{(1)})\actson v_1) \wedge\dots\wedge
 ((a_{(m+1)}b_{(m+1)})\actson v_{m+1})
 \tensor (ab)_{(m+2)} \\
 &= (S(a_{(1)},b_{(1)},v_1)+T(a_{(1)},b_{(1)},v_1)) \wedge\dots\\
 &\qquad\qquad\dots\wedge
 (S(a_{(m+1)},b_{(m+1)},v_{m+1})+T(a_{(m+1)},b_{(m+1)},v_{m+1}))
 \tensor (ab)_{(m+2)}
\ .
\end{align*}
Now we can simplify the wedge product using the distributive law and after swapping wedge factors and relabeling $v_1,\dots,v_{m+1}$ as $v'_1,\dots,v'_{m+1}$ when necessary, every summand will contain a factor
\[ S(a_{(1)},b_{(1)},v'_1) \wedge\dots\wedge S(a_{(i+1)},b_{(i+1)},v'_{i+1})
\]
or a factor
\[ T(a_{(1)},b_{(1)},v_1) \wedge\dots\wedge T(a_{(j+1)},b_{(j+1)},v'_{j+1})
\ ,
\]
so every summand vanishes if $a\in K'_i$ and $b\in K'_j$. Hence for such $a$ and $b$, the product $ab$ lies in $K'_m$, which implies $K_i K_j\subset K_{i+j}$.

Finally, note that $K_d=K'_d=H$, where $d=\dim V$.
\end{proof}

\begin{lem} $K_i$ is a subcoalgebra of $H$ and a submodule of $H$ under the adjoint action for all $i\geq 0$.
\end{lem}

\begin{proof} Consider $h\in K_i$. We can write $\Delta h=\sum_k r^k\tensor h^k$ for linearly independent $(h^k)_k$ in $H$ and suitable elements $(r^k)_k$ in $K'_i$. For a given index $j$, let $p_j$ be a projection of $H$ onto $\FF h^j$ along $h^k$ for all $k\neq j$. Then
\begin{align*}
 & (r^j_{(1)}\actson v_1)\wedge \dots \wedge (r^j_{(i+1)}\actson v_{i+1})\tensor r^j_{(i+2)} \tensor h^j
 \\
 &=
 (\id\tensor\id\tensor p_j)
 (\sum_k (r^k_{(1)}\actson v_1)\wedge \dots \wedge (r^k_{(i+1)}\actson v_{i+1})\tensor r^k_{(i+2)} \tensor h^k)
 \\
 &=
 (\id\tensor\id\tensor p_j)
 (\sum_k (r^k_{(1)}\actson v_1)\wedge \dots \wedge (r^k_{(i+1)}\actson v_{i+1})\tensor h^k_{(1)} \tensor h^k_{(2)})
 = 0
 \ ,
\end{align*}
so $\Delta r^j\in K'_i\tensor H$, so $r^j\in K_i$, and $K_i$ is a subcoalgebra, as desired.

To see that $K_i$ is a submodule of $H$, we first note that for all $h,k\in H$ and all $v\in V$,
\[ (k \cdot h)\actson v
 = (k_{(1)} h Sk_{(2)})\actson v
 = k_{(1)} \cdot (h\actson (Sk_{(2)} \cdot v))
\ .
\]
So assume $h\in K_i$. Then
\begin{align*}
 & ((k \cdot h)_{(1)}\actson v_1) \wedge \dots \wedge ((k \cdot h)_{(i+1)}\actson v_{i+1}) \tensor (k \cdot h)_{(i+2)} 
 \\
 &= 
 (k_{(1)} \cdot (h_{(1)}\actson (Sk_{(2)}\cdot v_1))) \wedge \dots
 \wedge (k_{(2i+1)} \cdot (h_{(i+1)}\actson (Sk_{(2i+2)}\cdot v_{i+1})))
 \tensor k_{(2i+3)} \cdot h_{(i+2)} 
 \\
 &=  k_{(1)} \cdot ((h_{(1)}\actson (Sk_{(2)}\cdot v_1)) \wedge \dots
 \wedge (h_{(i+1)}\actson (Sk_{(i+2)}\cdot v_{i+1})) )
 \tensor k_{(i+3)} \cdot h_{(i+2)}  
 \\
 &= 0
\end{align*}
and indeed, $k \cdot h\in K_i$.
\end{proof}

Analogous to the proof of \cite[Prop~2.8]{EGG} we define the notation
\begin{equation}
 (v_1,\dots,v_k|x,y)
 :=(\kappa(x,y)_{(1)} \actson v_1)\wedge\dots\wedge
 (\kappa(x,y)_{(k)}\actson v_k)\tensor\kappa(x,y)_{(k+1)}
 \in\Lambda^k V\tensor H
\end{equation}
for all $v_1,\dots,v_k,x,y\in V$.

Now we have a counterpart to \cite[Prop.~2.8]{EGG} on the ``support'' of $\kappa$:

\begin{prop} \label{image-kappa} Assume $\kappa:V\wedge V\to H$ has the Jacobi property. Then $\im\kappa\subset K_2$.
\end{prop}

\begin{proof} This is a word-for-word translation of \cite[Prop.~2.8]{EGG} and the associated lemmas:

Note that in $A$,
\[ [h,v]=(h_{(1)} \cdot v) h_{(2)}-\eps(h_{(1)})v h_{(2)}=(h_{(1)}\actson v) h_{(2)}
\ ,
\]
so using our new notation, the Jacobi identity reads
\begin{equation} \label{eq-Jacobi-rank1}
 (v|x,y)+(x|y,v)+(y|v,x)=0
 \tforall v,x,y\in V
 \ .
\end{equation}
Now as in \cite[Lem.~2.10]{EGG}, this implies
\begin{equation} \label{eq-Jacobi-rank2}
 (z,u|x,y)=(x,y|z,u)
 \tforall z,u,x,y\in V
 \ . 
\end{equation}
Now as in \cite[Lem.~2.11]{EGG}, this implies
\begin{equation} \label{eq-Jacobi-rank3}
(z,u,v|x,y)=0
\tforall z,u,v,x,y\in V
\ .
\end{equation}

Hence if $h=\kappa(x,y)\in H$ for elements $x,y\in V$, then
\[
 (h_{(1)}\actson z)\wedge (h_{(2)}\actson u)\wedge (h_{(3)}\actson v)
 \tensor h_{(4)} = 0
 \tforall z,u,v\in V
 \ ,
\]
so $\Delta h\in K'_2\tensor H$ and hence $h\in K_2$.
\end{proof}

\begin{rem} Compare this with \cite[Prop.~4.3]{Kh} which is formulated for a cocommutative bialgebra and with an additional deformation parameter $\l$ (and note that the above proof works for a cocommutative bialgebra, as well).
\end{rem}

Extending the class of examples we obtain from transferring the discussion in \cite[Sec.~2.3]{EGG} to our setting, we have the following class of examples:

\begin{defn}\label{defn-kappa-std} Consider elements $\tau\in (V\wedge V)^*\tensor K_0$,
\[
 \sigma=\sum_m \sigma_m\tensor h^m
 \in (V\wedge V)^*\tensor K_1
 \ ,\qquad
 \theta=\sum_i \theta_i\tensor k^i \in (V\wedge V)^*\tensor K_2
\ ,
\]
which can be viewed as linear maps from $V\wedge V$ to $K_0$, $K_1$ and $K_2$, respectively.
Using those we define new linear maps from $V\wedge V$ to $H$: $\kappa_\tau(x,y):=\tau(x,y)$,
\[ \kappa_\sigma(x,y):= \sum_m
    \sigma_m (h^m_{(1)}\actson x,y)h^m_{(2)}
 	+ \sigma_m (x,h^m_{(1)}\actson y)h^m_{(2)}
\ ,
\]
\[
  \kappa_\theta(x,y):= \sum_i \theta_i(k^i_{(1)}\actson x,k^i_{(2)}\actson y) k^i_{(3)}
  \
\]
for all $x,y\in V$, and
\begin{equation} \label{eq-kappa-example}
 \kappa := \kappa_\tau + \kappa_\sigma + \kappa_\theta
 \ .
\end{equation}
\end{defn}

\begin{rem} $\kappa_\sigma$ and $\kappa_\theta$ actually only depend on $[\sigma]$ and $[\theta]$ in $K_1/K_0$ and $K_2/K_1$, respectively. This is, because if $h\in K_0$ and $k\in K_1$, then
\[ h_{(1)}\actson x \tensor h_{(2)}
  = h_{(1)}\actson y \tensor h_{(2)}
  = 0
\]
and
\[ (k_{(1)}\actson x)\wedge (k_{(2)}\actson y) \tensor k_{(3)} = 0
\ .
\]
\end{rem}

\begin{lem} \label{lem-kappa-H-linear} Each of $\kappa_\tau$, $\kappa_\sigma$ or $\kappa_\theta$ as in the definition is $H$-equivariant if the corresponding map $\tau$, $\sigma$ or $\theta$ is $H$-equivariant, respectively.
In particular, $\kappa$ is $H$-equivariant if $\tau$, $\sigma$ and $\theta$ are $H$-equivariant.
\end{lem}

\begin{proof} For $\kappa_\tau$, the assertion is tautological. For $\kappa_\sigma,\kappa_\theta$ let us first note that for any $h,k\in H$ and any $x\in V$,
\[ h\actson (Sk \cdot x)
 = Sk_{(1)} \cdot  ((k_{(2)} h Sk_{(3)})\actson x)
 =  Sk_{(1)} \cdot  ((k_{(2)} \cdot h)\actson x)
\]
using the adjoint action in $H$. Now a linear map from $V\wedge V$ to $H$ is $H$-equivariant, if the corresponding element in $(V\wedge V)^*\tensor H$ is $H$-invariant. So we can verify for any $h\in H,x,y\in V$:
\begin{align*}
 (h \cdot \kappa_\theta)(x,y)
 &= \sum_i \theta_i(k^i_{(1)}\actson (Sh_{(1)} \cdot x), k^i_{(2)}\actson (Sh_{(2)} \cdot y)) h_{(3)} \cdot k^i_{(3)} \\
 &= \sum_i \theta_i(Sh_{(1)} \cdot (h_{(2)} \cdot k^i_{(1)})\actson x, Sh_{(3)} \cdot (h_{(4)} \cdot k^i_{(2)})\actson y) h_{(5)} \cdot k^i_{(3)} \\
 &= \sum_i (h_{(1)} \cdot \theta_i)((h_{(2)} \cdot k^i)_{(1)}\actson x, (h_{(2)} \cdot k^i)_{(2)}\actson y) (h_{(2)} \cdot k^i)_{(3)} \\
 &= \kappa_{h \cdot \theta}(x,y)
 \ ,
\end{align*}
and analogously for $\kappa_\sigma$.
\end{proof}

\begin{rem} Obviously, one way of obtaining $H$-equivariant $\tau,\sigma,\theta$ is by choosing $H$-invariant elements in $(V\wedge V)^*$ and $H$-invariant (that is, $H$-central) elements in $K_0$, $K_1$ and $K_2$.
The map $\kappa$ generated according to \Cref{defn-kappa-std} will be $H$-equivariant and will have the Jacobi property, so $A_{H,V,\kappa}$ will be a PBW deformation. If additionally $V$ is an orthogonal $H$-module, $A_{H,V,\kappa}$ will be a Hopf--Hecke algebra.
\end{rem}

\begin{prop} \label{kappa-example} Let $\kappa$ be as in \Cref{defn-kappa-std}. Then it has the Jacobi property.

In particular, if additionally $\tau$, $\sigma$, $\theta$ are $H$-equivariant, then $A=A_{H,V,\kappa}$ has the PBW property.
\end{prop}

\begin{proof} As in \cite[Thm.~2.13]{EGG}: By \Cref{prop-Jacobi-identity}, the PBW property is equivalent to the Jacobi identity if $\kappa$ is $H$-equivariant.

To verify the Jacobi property, we consider elements $x,y,z\in V$. Recall that the Jacobi identity reads
\[ 0 = (\kappa(x,y)_{(1)}\actson z) \kappa(x,y)_{(2)}
 + (\kappa(y,z)_{(1)}\actson x) \kappa(y,z)_{(2)}
 + (\kappa(z,x)_{(1)}\actson y) \kappa(z,x)_{(2)}
 \ .
\]

Now for all $h\in K_0$ and all $v\in V$,
\[ 0 = (h_{(1)}\actson v) \tensor h_{(2)}
\ ,
\]
which verifies the Jacobi identity for $\kappa_\tau$.

Also, for every index $m$ and all $x,y,z\in V$,
\[
0 = (h^m_{(1)}\actson x)\wedge (h^m_{(2)}\actson y)\wedge z \tensor h^m_{(3)}
 \ ,
\]
because $h^m\in K_1$, so
\begin{align*}
0
&= \sigma_m(h^m_{(1)}\actson x,h^m_{(2)}\actson y) z \tensor h^m_{(3)}
 + \sigma_m(h^m_{(1)}\actson y,z)(h^m_{(2)}\actson x) \tensor h^m_{(3)}
 + \sigma_m(z,h^m_{(1)}\actson x)(h^m_{(2)}\actson y) \tensor h^m_{(3)} \\
&= \sigma_m(h^m_{(1)}\actson y,z)(h^m_{(2)}\actson x) \tensor h^m_{(3)}
 + \sigma_m(z,h^m_{(1)}\actson x)(h^m_{(2)}\actson y) \tensor h^m_{(3)}
\ ,
\end{align*}
again, because $h^m\in K_1$.

Thus,
\begin{align*}
0 &= \sigma_m(h^m_{(1)}\actson x,y)(h^m_{(2)}\actson z)h^m_{(3)}
	+\sigma_m(x,h^m_{(1)}\actson y)(h^m_{(2)}\actson z)h^m_{(3)} \\
 &+\sigma_m(h^m_{(1)}\actson z,x)(h^m_{(2)}\actson y)h^m_{(3)}
	+\sigma_m(z,h^m_{(1)}\actson x)(h^m_{(2)}\actson y)h^m_{(3)} \\
 &+\sigma_m(h^m_{(1)}\actson y,z)(h^m_{(2)}\actson x)h^m_{(3)}
	+\sigma_m(y,h^m_{(1)}\actson z)(h^m_{(2)}\actson y)h^m_{(3)}
\ ,
\end{align*}
which verifies the Jacobi identity for $\kappa_\sigma$.

Finally for every index $i$ and all $x,y,z\in V$,
\[
 0 = (k^i_{(1)}\actson x)
   \wedge (k^i_{(2)}\actson y)
   \wedge (k^i_{(3)}\actson z) \tensor k^i_{(4)}
   \ ,
\]
because $k^i\in K_2$, so
\[
 0 = (
  \theta_i(k^i_{(1)}\actson x,k^i_{(2)}\actson y)
   (k^i_{(3)}\actson z)
  +\theta_i(k^i_{(1)}\actson z,k^i_{(2)}\actson x)
   (k^i_{(3)}\actson y)
  +\theta_i(k^i_{(1)}\actson y,k^i_{(2)}\actson z)
   (k^i_{(3)}\actson x)) k^i_{(4)}
   \ ,
\]
which verifies the Jacobi identity for $\kappa_\theta$.
\end{proof}

\begin{cor} In the situation of the proposition, if additionally $V$ is an orthogonal $H$-module, then $A=A_{H,V,\kappa}$ is a Hopf--Hecke algebra.
\end{cor}

\begin{defn} \label{def-standard} We call a PBW deformation $A=A_{H,V,\kappa}$ with a deformation map $\kappa$ as in \Cref{defn-kappa-std} a \emph{standard PBW deformation} or, if additionally $V$ is an orthogonal $H$-module, a \emph{standard Hopf--Hecke algebra}.
\end{defn}

We will investigate conditions under which PBW deformations or Hopf--Hecke algebras are standard.

\subsection{Maps with the Jacobi property for pointed cocommutative Hopf algebras}
\label{subsec-Jacobi-maps}
In the following, we consider the case of a pointed cocommutative Hopf algebra $H$ over $\FF$ (a field of characteristic $0$). We recall that this includes all cocommutative Hopf algebras over $\CC$.

Let $H$ be a cocommutative pointed Hopf algebra. Recall that by the structure theorem for cocommutative pointed Hopf algebras over a field of characteristic $0$, $H=H^1\rtimes\FF[G(H)]$, where $H^1$ is the universal enveloping algebra of the Lie algebra of primitive elements in $H$ and $\FF[G(H)]$ is the group algebra of the group of group-like elements $G(H)$ in $H$. For each group-like element $g\in G(H)$, $H^1g$ is a subcoalgebra of $H$ and $H=\bigoplus_{g\in G(H)} H^1g$ as coalgebras. Let $p_g:H\to H^1g$ be the corresponding projection map.

We continue to assume that $V$ is a finite-dimensional $H$-module. For any linear map $\kappa:V\wedge V\to H$ and any $g\in G(H)$, we define $\kappa_g:=p_g\circ\kappa:V\wedge V\to H^1 g$.

\begin{lem} \label{lem-kappa-g} A linear map $\kappa:V\wedge V\to H$ has the Jacobi property if and only if $\kappa_g$ has the Jacobi property for all $g\in G(H)$.
\end{lem}

\begin{proof} We can apply $\id_V\tensor p_g$ to the Jacobi identity in $V\tensor H$ to obtain the Jacobi identity for $\kappa_g$.
\end{proof}

\begin{defn} Let $C$ be a coalgebra. A filtration $(C_k)_k$ of $C$ as vector space is called a \emph{coalgebra filtration} if
\[
\Delta C_k\subset \sum_{0\leq i\leq k} C_i\tensor C_{k-i}
\ .
\]
Let $C_0$ be the \emph{coradical} of $C$, i.e. the sum of all simple subcoalgebras of $C$. The \emph{coradical filtration} of $C$ is defined inductively by $C_{k+1}:=\Delta\inv(C_0\tensor C+C\tensor C_k)$.
\end{defn}

We recall well-known facts from the theory of coalgebras:
The coradical filtration is a coalgebra filtration such that $C=\bigcup_{k\geq0} C_k$ for every coalgebra $C$.
If $C$ is a pointed coalgebra, for instance any cocommutative coalgebra over $\CC$, then $C_0=\bigoplus_{g\in G(C)} \FF g$ for the set of group-like elements $G(C)$ in $C$.

\newcommand{\rnk}{\operatorname{rnk}}
In the following, $\rnk$ will denote the rank of the action of an element of $H$ acting on $V$. We record a useful lemma. 

\begin{lem} We consider an element $g\in G(H)$ with $\rnk(g-1)=1$. Then $g$ acts diagonalizably on $V$ if either $g$ has finite order or if $V$ is an orthogonal module (i.e., there is a non-degenerate $H$-invariant symmetric bilinear form $\tri$ on $V$).
\end{lem}

\begin{proof} Since $\rnk(g-1)=1$, we can write $(g-1)|_V = f(\cdot)v$ with suitable non-zero $f\in V^*,v\in V$. Now it is enough to show $f(v)\neq0$, because then a basis of the kernel of $(g-1)|_V$ together with $v$ form a basis of $V$ consisting of eigenvectors of $g$.

If $g$ has finite order $r$, assume $f(v)=0$, then 
\[ \id_V = g|_V^r = (\id_V+f(\cdot)v)^r = \id_V + r f(\cdot) v
\ ,
\]
which is a contradiction. Hence $f(v)\neq0$, and $g$ acts diagonalizably.

Similarly, assume $V$ is orthogonal and $f(v)=0$. Since $f\neq 0$, we can pick $x\in V$ such that $f(x)\neq 0$, and we obtain
\[ \tri[v,x] = \tri[gv,gx] = \tri[v,x+f(x)v]
\Rightarrow \tri[v,v] = 0
\ .
\]
Now for all $y\in V\setminus(\ker f)$, $z\in\ker f$,
\begin{gather*}
     \tri[y,y] = \tri[gy,gy] = \tri[y+f(y)v,y+f(y)v]
 \Rightarrow \tri[v,y] = 0
\ ,
\\
\tri[x,z] = \tri[gx,gz] = \tri[x+f(x)v, z] 
 \Rightarrow \tri[v,z] = 0
 \ .
\end{gather*}
But this means that $\tri[v,V]=0$, which is a contradiction. Hence $f(v)\neq 0$ and again, $g$ acts diagonalizably.
\end{proof}

We have the following information on the group-like elements $g$ which are necessary to determine $\kappa$ and the corresponding maps $\kappa_g$ (see also \cite[Sec.~1]{RS}, \cite[Sec.~2.3]{EGG}):
\begin{prop} \label{rank-g} Let $\kappa:V\wedge V\to H$ be a linear map with the Jacobi property. Then the following holds for every $g\in G(H)$, where $(g-1)$ denotes the corresponding operator on $V$:
\begin{itemize}
\item{$\kappa_g=0$ if $\rnk(g-1)\not\in\{0,1,2\}$.
}
\item{If $\rnk(g-1)=1$, then $\kappa_g(x,y)=0$ for all $x,y\in V$ satisfying $((g-1)\cdot x)\tensor y-((g-1)\cdot y)\tensor x=0$.
}
\item{If $\rnk(g-1)=1$ and $g$ acts diagonalizably on $V$ (for instance, if $g$ has finite order or $V$ is an orthogonal $H$-module), then $\kappa_g(x,y)=0$ for all $x,y\in V$ satisfying $((g-1)\cdot x)\wedge y+x\wedge ((g-1)\cdot y)=0$.
}
\item{If $\rnk(g-1)=2$, then $\kappa_g(x,y)=0$ for all $x,y\in V$ satisfying $((g-1) \cdot x)\wedge ((g-1) \cdot y)=0$.
}
\end{itemize}
\end{prop}

\begin{proof} We fix $g\in G(H)$. Then by \Cref{lem-kappa-g}, $\kappa_g$ has the Jacobi property, so it is enough to consider the case $\kappa=\kappa_g$.

It is a basic statement on coalgebras that every finite-dimensional subspace is contained in a finite-dimensional subcoalgebra. Let $C$ be such a finite-dimensional subcoalgebra of $H^1 g$ (which is a subcoalgebra of $H$) containing $(\im\kappa_g)$. Let $(C_k)_{k\geq0}$ be the coradical filtration of $C$ and let $k$ be minimal such that $\im\kappa\subset C_k$. Note that $C_0=\FF g$ now, because $g$ is the unique group-like element in $C$.

Then we can write $\kappa_g=\sum_i \theta_i h^i$ with suitable non-zero $(\theta_i)_i$ in $(V\wedge V)^*$ and linearly independent $(h^i)_i$ in $C_k$. Let $J$ be the set of indices $j$ such that $h^j\in C_k\setminus C_{k-1}$ (where we set $C_{-1}=0$). Since $k$ was chosen minimally, $J\neq\emptyset$. For every $j\in J$, let $p_j$ be a projection of $C_k$ onto $\FF h^j$ along $C_{k-1}$ and along $h^i$ for all $i\neq j$. Then
\[ (\id\tensor p_j)\circ\Delta(h^i) = \delta_{ij} g\tensor h^j
\tforall i
\ .
\]

Thus if we apply $(\id\tensor p_j)$ to \eqref{eq-Jacobi-rank3}, this yields
\[ 0 = (g-1) \cdot z \wedge (g-1) \cdot u \wedge (g-1) \cdot v \tensor \theta_j(x,y) h^j
\tforall z,u,v,x,y\in V
\ ,
\]
so the operator $(g-1)$ has rank at most 2.

If we apply $(\id\tensor p_j)$ to the Jacobi identity $0=(x|y,z)+(y|z,x)+(z|x,y)$ in $V\tensor H$ for any $x,y,z\in V$, we obtain
\[
0
 = (((g-1) \cdot x)\theta_j(y,z) + ((g-1) \cdot y)\theta_j(z,x) + ((g-1) \cdot z)\theta_j(x,y)) \tensor h^j
 \ .
\]
Let us assume that $(g-1)$ has rank $1$, and let us pick $f\in V^*$ and $z\in V$ such that $f((g-1)\cdot z)=1$. Then the last equation implies
\[
\theta_j(x,y) = f((g-1)\cdot z)\theta_j(x,y)
 = - (f\tensor\theta_j(\cdot,z)) (((g-1)\cdot x) \tensor y-((g-1)\cdot y)\tensor x)
 \ ,
\]
so that $\theta_j(x,y)=0$ if $((g-1) \cdot x)\tensor y - ((g-1) \cdot y)\tensor x=0$.

If additionally $g$ acts diagonalizably, then $(g-1)\cdot v=f((g-1)\cdot v)z$ for all $v\in V$, so
\[
\theta_j(x,y)
 = -\theta_j(y,(g-1)\cdot x)+\theta_j(x,(g-1)\cdot y)
 = \theta_j((g-1)\cdot x\wedge y+x\wedge (g-1)\cdot y)
 \ ,
\]
which confirms that $\theta_j(x,y)=0$ if $(g-1)\cdot x\wedge y+x\wedge (g-1)\cdot y=0$.

Let us assume that $(g-1)$ has rank $2$. We apply $(\id\tensor p_j)$ to \eqref{eq-Jacobi-rank2} to obtain
\[ (g-1) \cdot z \wedge (g-1) \cdot u \tensor \theta_j(x,y)
 = (g-1) \cdot x \wedge (g-1) \cdot y \tensor \theta_j(z,u)
\]
for all $z,u,x,y\in V$. Since $(g-1)$ has rank $2$, we can pick $z,u$ such that $(g-1) \cdot z\wedge(g-1) \cdot u$ is non-zero. So $\theta_j(x,y)$ has to be zero if $(g-1) \cdot x\wedge(g-1)y=0$.

Hence $\theta_j$ has to vanish on the subspaces as stated for every $j\in J$. Hence
\[
 \kappa(x,y)
  =\sum_{i\not\in J} \theta_i(x,y) h^i
  =:\kappa'(x,y)
\]
on these subspaces, but $\im\kappa'\subset C_{k-1}$. We repeat the argument inductively replacing $\kappa$ by $\kappa'$ each time until $\im\kappa'\subset C_{-1}=0$.
\end{proof}

To compare this with the classical situation of $H$ being the group-algebra of a finite group, we note:

\begin{cor} \label{cor-image-of-k1} Let $\kappa:V\wedge V\to H$ be an $H$-equivariant $\FF$-linear map with the Jacobi property, and fix $g\in G(H)$ such that $\rnk(g-1)=1$ and $g$ acts diagonalizably on $V$ (which is true for instance if $g$ has finite order or $V$ is an orthogonal $H$-module). Let $r$ be the non-zero eigenvalue of $(g-1)$. Then
\[ \im\kappa_g \subset \{x\in H^1 g: gxg\inv=(r+1)x\}
\ .
\]

In particular, if $H$ is the group algebra of a finite-group, then $\kappa_g=0$ for all $g$ with $\rnk(g-1)=1$.
\end{cor}

\begin{proof} Let $v\in V$ be an eigenvector of $(g-1)$ with eigenvalue $r\in\FF\setminus\{0\}$ such that $V=\FF v\oplus\ker(g-1)$. Now $\kappa_g(x,y)=0$ for all $x,y\in \ker(g-1)$ and $\kappa_g(x,y)=0$ for all $x,y\in\FF v$, because in both cases,
\[ (g-1)\cdot x\wedge y+x\wedge (g-1)\cdot y=0
\ .
\]
Assume $x=v$ and $y\in\ker(g-1)$. Then due to $H$-equivariance,
\[ g\kappa_g(x,y)g\inv = \kappa_g(g \cdot x,g \cdot y)= (r+1)\kappa_g(x,y)
\ ,
\]
so indeed $\im\kappa_g$ lies in the subspace of $H^1 g$ on which $g$ acts by $(r+1)$.

If $H$ is the group algebra of a finite group, then $H^1 g=\FF g$, so $g$ acts trivially on $H^1 g$, but $r+1\neq 1$.
\end{proof}

\begin{defn} For every $p\geq0$ and a linear map $\kappa:V\wedge V\to H$, we define
\[ \kappa_{(p)}:=\sum_{g\in G(H),\rnk(g-1)=p,\im\kappa_g\subset K_p} \kappa_g
\ .
\]
\end{defn}

We observe that if $\kappa$ has the Jacobi property, by \Cref{image-kappa} and \Cref{rank-g}, $\kappa_{(p)}=0$ for $p>2$ and the condition $\im\kappa_g\subset K_2$ in the definition of $\kappa_{(2)}$ is redundant. We also note that if $\kappa$ has the Jacobi property, then $\kappa_{(p)}$ has the Jacobi property for every $p\geq0$ by \Cref{lem-kappa-g}, since $\kappa_{(p)}$ is a sum of $\kappa_g$'s.
 
\begin{lem} \label{kappa-0} For every $\kappa:V\wedge V\to H$ with the Jacobi property, $\kappa_{(0)}$ is of the form of \Cref{defn-kappa-std}.
\end{lem}

\begin{proof} This is immediate from the definition of $\kappa_{(0)}$.
\end{proof}

\begin{prop} \label{kappa-1} Assume $G(H)$ is a torsion group (for instance, a finite group) or $V$ is an orthogonal $H$-module. Then for every $\kappa:V\wedge V\to H$ with the Jacobi property, $\kappa_{(1)}$ is of the form
\[ \kappa_{(1)}(x,y) = \sum_m
 \sigma_m(h^m_{(1)}\actson x,y)h^m_{(2)}
 +\sigma_m(x,h^m_{(1)}\actson y)h^m_{(2)}
\]
with $h^m$ in $K_1$ and $\sigma_m\in(V\wedge V)^*$ for every $m$.
In particular, it is of the form of \Cref{defn-kappa-std}.
\end{prop}

\begin{proof} By \Cref{lem-kappa-g}, it is enough to show the assertion for $\kappa=\kappa_g$ for a fixed $g\in G(H)$ with $\rnk(g-1)=1$ and such that $\im\kappa_g\subset K_1$.

\newcommand{\tisigma}{{\tilde\sigma}}
We can write $\kappa=\sum_i \sigma_i h^i$ with linearly independent $h^i$ in $H^1g\cap K_1$ and suitable $\sigma_i$ in $(V\wedge V)^*$. Let $J$ be the set of indices $j$ such that $h^j$ lies in maximal degree $d$ of the coradical filtration. Since $\rnk(g-1)=1$, by \Cref{rank-g} we know that
\[ \sigma_j(x,y) = \tisigma_j((g-1) \cdot x\wedge y+x\wedge (g-1) \cdot y)
\]
for some $\tisigma_j$ in $(V\wedge V)^*$. We define
\[ \kappa'(x,y):=\sum_{j\in J} \tisigma_j(h^j_{(1)}\actson x,y)h^j_{(2)}+\tisigma_j(x,h^j_{(1)}\actson y) h^j_{(2)}
\ ,
\]
then by \Cref{kappa-example}, $\kappa'$ has the Jacobi property, so $\kappa''=\kappa-\kappa'$ has the Jacobi property, but the image of $\kappa''$ lies in degree $\leq d-1$ of the coradical filtration, because the highest degree terms of $\kappa$ and $\kappa'$ cancel. We can replace $\kappa$ by $\kappa''$ and proceed inductively until the image of $\kappa''$ lies in degree $-1$, so $\kappa''=0$.
\end{proof}

Finally, for all $g\in G(H)$ with $\rnk(g-1)=2$, let us fix
$\theta_g\in(V\wedge V)^*$ which do not vanish on the one-dimensional spaces $(g-1)V\wedge(g-1)V$.

\begin{prop} \label{kappa-2} For every $\kappa:V\wedge V\to H$ with the Jacobi property, $\kappa_{(2)}$ is of the form
\[ \kappa_{(2)}(x,y) = \sum_{g\in G(H),\rnk(g-1)=2} \theta_g(h^g_{(1)}\actson x,h^g_{(2)}\actson y) h^g_{(3)}
\]
with $h^g$ in $H^1g\cap K_2$ for every $g$.
In particular, it is of the form of \Cref{defn-kappa-std}.
\end{prop}

\begin{proof} By \Cref{lem-kappa-g}, it is enough to show this for $\kappa=\kappa_g$ for a fixed $g\in G(H)$ with $\rnk(g-1)=2$.

Since $\rnk(g-1)=2$, the restriction of any skew-symmetric bilinear form on $V$ to $(g-1)V\wedge(g-1)V$ is just a scalar multiple of the restriction of $\theta_g$.

\newcommand{\titheta}{{\tilde\theta}}
We can write $\kappa=\sum_i \theta_i k^i$ with linearly independent $k^i$ in $H^1g\cap K_2$ and suitable $\theta_i$ in $(V\wedge V)^*$. Let $J$ be the set of indices $j$ such that $k^j$ lies in maximal degree $d$ of the coradical filtration. Since $\rnk(g-1)=2$, by \Cref{rank-g} we know that
\[ \theta_j(x,y) = \titheta_j((g-1) \cdot x,(g-1) \cdot y)
 = r_j \theta_g((g-1) \cdot x,(g-1) \cdot y)
\]
for some $\titheta_j$ in $(V\wedge V)^*$ and for some $r_j\in\FF$. We define $h^j:=r_j k^j$ and
\[ \kappa'(x,y):=\sum_{j\in J} \theta_g(h^j_{(1)}\actson x,h^j_{(2)}\actson y) h^j_{(3)}
\ ,
\]
then by \Cref{kappa-example}, $\kappa'$ has the Jacobi property, so $\kappa''=\kappa-\kappa'$ has the Jacobi property, but the image of $\kappa''$ lies in degree $\leq d-1$ of the coradical filtration, because the highest degree terms of $\kappa$ and $\kappa'$ cancel. We can replace $\kappa$ by $\kappa''$ and proceed inductively until the image of $\kappa''$ lies in degree $-1$, so $\kappa''=0$. This way we see that
\[ \kappa(x,y)=\sum_{p} \theta_g(h^p_{(1)}\actson x,h^p_{(2)}\actson y) h^p_{(3)}
\]
for some $(h^p)_p$ in $H^1g\cap K_2$, but now we can define $h^g:=\sum_p h^p$ and the assertion follows.
\end{proof}

\begin{defn} \label{def-hs} Let us denote the class of Hopf--Hecke algebras $A_{H,V,\kappa}$ by $\boldsymbol{H}$ and the class of standard Hopf--Hecke algebras by $\boldsymbol{S}$ (see \Cref{def-standard}), that is, the elements of $\boldsymbol{S}$ are deformations with deformation maps $\kappa$ of the form of \Cref{defn-kappa-std}. For every PBW deformation $A=A_{H,V,\kappa}$ (even if $V$ is not an orthogonal module) we define 
\[ 
\hs(\kappa) := \kappa_{(0)} + \kappa_{(1)} +\kappa_{(2)}
\qquad\text{and}\qquad 
\hs(A_{H,V,\kappa}):=A_{H,V,\hs(\kappa)}
\ .
\]
\end{defn}

In particular, $\hs$ can be applied to a Hopf--Hecke algebra $A=A_{H,V,\kappa}$, for which the $H$-module $V$ is orthogonal.
 
\begin{prop}\label{rem-hs} $\hs:\mathbf{H}\to\mathbf{H}$ is a well-defined idempotent mapping and $\hs(\mathbf{H})\subsetneq\mathbf{S}$.
\end{prop}

\begin{proof} To summarize \Cref{kappa-example}, \Cref{kappa-0}, \Cref{kappa-1} and \Cref{kappa-2}, for every $\kappa$ with the Jacobi property, $\hs(\kappa)=\kappa_{(0)}+\kappa_{(1)}+\kappa_{(2)}$ is of the form of \Cref{defn-kappa-std} and has the Jacobi property. In other words, for every PBW deformation $A=A_{H,V,\kappa}$, the deformation $\hs(A)=A_{H,V,\kappa_{(0)}+\kappa_{(1)}+\kappa_{(2)}}$ is a standard PBW deformation. Since the orthogonality of $V$ is unaffected by $\hs$, $\hs$ sends Hopf-Hecke algebras to standard Hopf--Hecke algebras. By its definition, $\hs$ is an idempotent mapping.

It remains to see that there are standard Hopf--Hecke algebras which cannot be obtained through $\hs$ from any Hopf--Hecke algebra. We will describe such an algebra: Let us pick non-zero $y\in\CC$ and $x\in\CC^*$ such that $x(y)=1$, then $H:=\gl_1(\CC)=\CC$ acts on $V=\CC\oplus\CC^*$, where $Ix=x$ and $Iy=-y$ for $I:=1\in\gl_1(\CC)$. Now $V$ carries a non-degenerate $H$-invariant symmetric bilinear form $\tri[\cdot,\cdot]$ defined by $\tri[x,x]=\tri[y,y]=0$ and $\tri[x,y]=1$, that is, $V$ is orthogonal. Similarly, a non-degenerate $H$-invariant skew-symmetric bilinear form $\theta_1\in(V\wedge V)^*$ is defined by $\theta_1(x\wedge y)=1$. In this situation, $H$ is abelian, $I\not\in K_0$ and $I^3\in K_2$. So $\theta:=\theta_1\o I^3:V\wedge V\to K_2$ is a well-defined $H$-linear map. Thus, by \Cref{defn-kappa-std} and \Cref{kappa-example}, we have a standard PBW deformation $A=A_{H,V,\kappa_\theta}$ (which, in fact, is isomorphic tox $\U(\sl_2(\CC))$), where
\[ \kappa_\theta(x\wedge y)
= \theta_1(I^3_{(1)}\actson x, I^3_{(2)}\actson y) I^3_{(3)} 
= \theta_1(I x, I y) I
= - I
\ .
\]
In particular, $\im\kappa\not\subset K_0$. Hence, this standard PBW deformation is not obtained from $\hs$, that is, $\hs$ is not surjective onto $\boldsymbol{S}$.
\end{proof}

In \Cref{sec-inf-Cherednik-alg} we will consider infinitesimal Cherednik algebras, which generalize the (counter)\-example appearing in the proof: There, instead of $\gl_1(\CC)$, we consider $\gl_n(\CC)$ for arbitrary $n\geq1$ and the deformation map $\kappa$ takes values not only in $\gl_n(\CC)\subset\U(\gl_n(\CC))$, but in all of $\U(\gl_n(\CC))$.

\begin{rem}  It might be another interesting question which maps $\kappa$ have the Jacobi property other than the ones of the form of \Cref{defn-kappa-std}, or similarly, which PBW deformations $A=A_{H,V,\kappa}$ are not standard PBW deformations.

Regarding the first question, note that by \Cref{lem-kappa-g} and \Cref{kappa-2}, it is enough to consider the case $\kappa=\kappa_g$ for a fixed group-like $g$ with $\rnk(g-1)\in\{0,1\}$, and by the results in \Cref{kappa-0}, and \Cref{kappa-1}, an example with orthogonal $V$ extending our partial characterization would necessarily satisfy $\im\kappa_g\not\subset K_{\rnk(g-1)}$.

If $H$ is the group-algebra of a finite group, there can be no such maps, because by \Cref{cor-image-of-k1}, $\kappa_g=0$ for all $g\in G(H)$ with $\rnk(g-1)=1$ and for all $g\in G(H)$ with $\rnk(g-1)=0$, $\im\kappa_g\subset H^1g=\FF g\subset K_0$ automatically, so $\kappa=\kappa_{(0)}+\kappa_{(2)}=\hs(\kappa)$. In particular, all PBW deformations are standard in this case.
\end{rem}


\section{Dirac cohomology for Hopf--Hecke algebras}
\label{sec-Dirac}
For the convenience of the reader we would like to recall some central notions and results from \cite{Fl} which will be used in the course of this paper.

We fix a cocommutative Hopf algebra $H$, an orthogonal (finite-dimensional) $H$-module $V$ with bilinear form $\tri$ and an $H$-equivariant $\FF$-linear map $\kappa:V\wedge V\to H$ with the Jacobi property (\Cref{defn-Jacobi-property}), so $A=A_{H,V\kappa}$ is a Hopf--Hecke algebra.  Since $V$ is fixed, we use the shorthand $C$ for the Clifford algebra, which can be defined (in characteristic not $2$) as a quotient of the tensor algebra $T(V)$ by
\[ C = C(V) := T(V) / (vw+wv-2\tri[v,w])
\ .
\]

For a general Hopf--Hecke algebra $A=A_{H,V,\kappa}$, we have the following definitions and results (\xcite[Sec.~3.2]{Fl}):
\begin{defn} \label{defn-Casimir-Dirac} Let $(v_k)_k$, $(v^k)_k$ be a pair of orthogonal bases of $V$ with respect to $\tri$. Then the \emph{Casimir element} $\Omega$ and the \emph{Dirac element} $D$ are defined to be
\begin{equation} \label{eq-Casimir-Dirac}
\Omega:= \sum_k v_k v^k
\in A
\ ,
\qquad
D:= \sum_k v_k \tensor v^k
\in A\tensor C
\ .
\end{equation}
\end{defn}

\begin{lem} \label{square-of-D} The Casimir and the Dirac element are independent of the choice of dual bases, they are $H$-invariant and
\[
 D^2 = \Omega\tensor1
   + \tfrac12 \sum_{k<l} \kappa(v_k,v_l) \tensor [v^k, v^l]
\]
in $A\tensor C$, where the commutator is taken in $C$.
\end{lem}

If $\FF=\CC$, then up to equivalence, there is a unique irreducible $C$-module if $\dim V$ is even, or two irreducible $C$-modules if $\dim V$ is odd. We fix an irreducible $C$-module $S$ and let $M$ be an $A$-module. Then $M\tensor S$ is an $A\tensor C$-module, and, in particular, the Dirac operator acts on $M\tensor S$.
\begin{defn} \label{defn-Dirac-cohomology} The \emph{Dirac cohomology of $M$} (with respect to $S$) is defined as
\[ H^D(M):=\ker D / (\im D\cap\ker D)
\ .
\]
\end{defn}

\begin{lem} \label{lem-D-diagonalizable} If $D$ acts diagonalizably on $M\tensor S$ (e.g., as a normal operator), then $H^D(M)\cong \ker D^2$.
\end{lem}

In general, the Dirac cohomology $H^D(M)$ is a module not necessarily of the Hopf algebra $H$, but of a certain Hopf algebra double cover $\tiH$ of $H$, which we call the \emph{pin cover}. In \cite{Fl} a result relating the Dirac cohomology with central characters (``Vogan's conjecture'') is proved under a certain condition regarding the square of the Dirac operator. If this condition is met, we call the Hopf--Hecke algebra a \emph{Barbasch--Sahi algebra}.

In the following we will be interested in the special case where $H=\U(\gfrak)$, the universal enveloping algebra of a finite-dimensional complex Lie algebra $\gfrak$. In particular, $H$ is pointed and $1\in H$ is the unique group-like element. This simplifies the pin cover construction and also the condition on $D^2$ significantly.

\begin{rem}\label{rem-bivectors} Let $\so(V)$ denote the Lie algebra of skew-symmetric linear operators of $V$ with respect to $\tri$ and let $\Biv(V)$ be the Lie subalgebra of the Clifford algebra $C(V)$ generated by the commutators $w_1 w_2-w_2 w_1$ in $C(V)$ for vectors $w_1,w_2\in V$. Then we have a Lie algebra isomorphism
\begin{equation}\label{eq-bivectors}
    \phi:\Biv(V)\to\so(V)\ ,\quad
    \tfrac12(w_1 w_2-w_2 w_1) \mapsto (v\mapsto -2\tri[w_1,v]w_2 + 2\tri[w_2,v]w_1)
    \ ,
\end{equation}
which can also be realized as taking commutators in $C$, leaving the subspace $V$ invariant. For more information on Clifford algebras and this isomorphism we refer to \cite{HP-book, Me}.
\end{rem}

Now we have the following concrete description of the pin cover of $H$ as constructed in \cite[Sec.~2]{Fl}:

\begin{defn} \label{defn-pin-cover-Ug} Let $\tiH:=H\oplus H$, let $\pi:\tiH\to H$ be the natural projection onto the first copy of $H$, and let $\gamma:\tiH\to C$ be the algebra map defined by $\tih\mapsto\phi\inv(\pi(\tih)\cdot)$ for all $\tih\in \gfrak\oplus H \subset\tiH$, where the element $\pi(\tih)$ of $H$ can be viewed as a skew-symmetric endomorphism of $V$, since $V$ is an orthogonal module.
\end{defn}

\begin{prop} $(\tiH,\pi,\gamma)$ is the pin cover of $H$ with respect to $V$ in the sense of \xcite[Def.~2.11]{Fl}, and it splits in the sense of \xcite[Def.~2.13]{Fl}, i.e.~the epimorphism $\pi:\tiH\to H$ splits as a Hopf algebra map.
\end{prop}

We recall (\xcite[Def.~2.5]{Fl}) that for a pointed cocommutative Hopf algebra with an orthogonal module we have an algebra $\ZZ_2$-gradation which assigns each group-like element the determinant of the corresponding operator on $V$ and each primitive element degree $1$ in $\ZZ_2\cong\{\pm1\}$. Obviously, in our setting where $H=\U(\gfrak)$, $H=H\ieven$ and $\tiH=\tiH\ieven$ irrespective of the module $V$.

We also recall the definition (\xcite[Def.~2.11]{Fl}) of the \emph{diagonal map},
\[ \DeltaC:\tiH\to H\tensor C,\qquad \tih\mapsto \pi(\tih_{(1)})\tensor \gamma(\tih_{(2)})
\ ,
\]
and of $H':=\tiH/\ker\DeltaC$. Since the pin cover splits by our construction, we have $H'\cong H$, we can consider $H$ as a Hopf subalgebra of $\tiH$ and we have algebra maps $\gamma|_H:H\to C,h\mapsto\phi\inv(h\cdot),$ and $\DeltaC|_H:H\to H\tensor C,h\mapsto h_{(1)}\tensor\gamma|_H(h_{(2)}),$ which we denote by $\gamma$, $\DeltaC$, as well (abusing notation).

We now have
\begin{lem}[{\cite[Lem.~3.9]{Fl}}] \label{lem-D-Delta-commute} $D$ and $\DeltaC(h)$ commute in $A\tensor C$ for all $h\in H$.
\end{lem}
Consequently, $H^D(M)$ is an $H$-module.

Finally, we recall that the Hopf--Hecke algebra $A$ defined by $(H,V,\kappa)$ is called a \emph{Barbasch--Sahi algebra} if $D$ satisfies the \emph{Parthasarathy condition},
\[
 D^2 \in Z(A\tensor C)+\Delta_C(\tiH\ieven)
 \ .
\]
Now with the Hopf algebra $H$ as above, this is equivalent to
\[
 D^2 \in Z(A\tensor C)+\DeltaC(H)
 \ .
\]

\begin{defn} \label{def-B} Let us denote the class of Barbasch--Sahi algebras by $\boldsymbol{B}$.
\end{defn}

The significance of this class of algebras is that in this case one has a ``good'' notion of Dirac cohomology, with consequences for the representation theory. More precisely, the non-vanishing of Dirac cohomology imposes a strong restriction on a representation, in particular its central character is uniquely determined by its Dirac cohomology.

\begin{rem} By definition, $\boldsymbol{B}\subset\boldsymbol{H}$. However, this inclusion is proper in general. For instance, the rational Cherednik algebra with parameters $t,c$ is a Hopf--Hecke algebra with $H=\CC[W]$, the group algebra of a reflection group. It is explained in \cite{EGG} that this definition is standard according to our \Cref{def-standard}. In \cite[Prop.~4.9, Rem.~4.10]{Ci}, the square of the Dirac operator is computed, and it is observed that for $t\neq 0$, it is not of the form required to make it a Barbasch--Sahi algebra.

\label{rem-hs-diag} Combining this with \Cref{rem-hs} we obtain the following diagram, where it is an open question, if the vertical inclusions are proper:
\hsdiag

\end{rem}


\section{Infinitesimal Cherednik algebras of \texorpdfstring{$\GL_n$}{GL\_n}}
\label{sec-inf-Cherednik-alg}

\subsection{Motivation}
We fix a cocommutative Hopf algebra $H$ over $\CC$ and a completely reducible $H$-module $V$. Let us recall a well-known characterization of modules with both a symmetric and a skew-symmetric form. We give the proofs for completeness.

\begin{prop} If $V$ admits both a symmetric and a skew-symmetric non-degenerate $H$-invariant bilinear form, then $V$ is of the form $V\cong W\oplus W^*$ for an $H$-module $W$.
\end{prop}

\begin{proof} Since $V$ is completely reducible, we can decompose $V$ as a direct sum of simple submodules, and we can group these simple submodules such that
\[ V = \bigoplus_{i=1}^k V_i^{a_i} \oplus \bigoplus_{j=1}^m W_j^{b_j} \oplus (W_j^*)^{c_j}
\]
with positive integers $(a_i)_i, (b_j)_j, (c_j)_j$ and self-dual modules $(V_i)_i$ and such that $(V_i)_i,(W_j)_j,(W_j^*)_j$ are all pairwise non-isomorphic simple $H$-modules. As $V$ admits a non-degenerate $H$-invariant bilinear form, it is self-dual, so $b_j=c_j$ for each $j$. Hence, it is enough to show that $a_i$ is even for each $i$.

Consider two simple submodules $V'$ and $V''$ of $V$ and let $\alpha$ be a non-degenerate $H$-invariant bilinear form on $V$. Then $v\mapsto \alpha(\cdot,v)$ is an $H$-linear map from $V'$ to $(V'')^*$, but since $V'$ and $V''$ are simple, the map has to be an isomorphism or $0$. Hence the restriction of $\alpha$ to $V_i^{a_i}$ has to be non-degenerate for each $i$. This means that $V_i^{a_i}$ admits both a symmetric and a skew-symmetric non-degenerate $H$-invariant bilinear form for each $i$.

We consider a fixed index $i$. Since $V_i$ is self-dual, there is an $H$-linear isomorphism $V_i\to V_i^*$ or, equivalently, a non-degenerate $H$-invariant bilinear form $\alpha$ on $V_i$. We can view $\alpha$ as the sum of a symmetric and a skew-symmetric bilinear form, and since $\alpha$ is $H$-invariant, both summands have to be $H$-invariant, as well. Since $V_i$ is simple, the space of $H$-linear endomorphisms, equivalently, $H$-invariant bilinear forms is one-dimensional. Hence $\alpha$ has to be symmetric (case a) or skew-symmetric (case b).

We write $V_i^{a_i}=V_i\tensor \CC^{a_i}$ and we pick a basis $(e_p)_{1\leq p\leq a_i}$ of $\CC^{a_i}$. Let $\beta$ be a non-degenerate $H$-invariant skew-symmetric (case a) or symmetric (case b) bilinear form on $V_i^{a_i}$. Now for every $1\leq p,q\leq a_i$, the map $(v,v')\mapsto \beta(v\tensor e_p,v'\tensor e_q)$ is an $H$-invariant bilinear form on $V_i$, so it has to be a multiple of $\alpha$. Hence $\beta(v\tensor e_p,v'\tensor e_q)=\gamma(e_p,e_q) \alpha(v,v')$ for scalars $(\gamma(e_p,e_q))_{p,q}$, which defines a bilinear form $\gamma$ on $\CC^{a_i}$. For $\beta$ to be skew-symmetric (case a) or symmetric (case b), $\gamma$ has to be skew-symmetric. Now if $a_i$ is odd, $\gamma$ cannot be non-degenerate, so there is a vector $e\in\CC^n$ such that $\gamma(e,e')=0$ for all $e'\in\CC^n$, and consequently, $\beta(v\tensor e,v'\tensor e')=0$ for all $v,v'\in V_i$ and $e'\in\CC^n$. This is a contradiction, since $\beta$ was assumed to be non-degenerate. Hence $a_i$ has to be even, which was to be shown.
\end{proof}

\begin{prop} \label{modules-with-two-forms} The completely reducible finite-dimensional $H$-modules $V$ which admit both a symmetric and a skew-symmetric non-degenerate $H$-invariant bilinear form are exactly the $H$-modules of the form $V\cong W\oplus W^*$ for finite-dimensional $H$-modules $W$.
\end{prop}

\begin{proof} It only remains to show that modules of the form $W\oplus W^*$ admit forms as required. Let $(\cdot,\cdot):W^*\tensor W\to\CC$ be the natural pairing. By definition of the contragredient action of $H$ on $W^*$, the pairing is $H$-invariant. We define the forms $\alpha,\beta$ by
\[ \alpha(y+x,y'+x'):=(y,x')+(y',x)
\ ,\qquad
 \beta(y+x,y'+x'):=(y,x')-(y',x)
 \ .
\]
Then since $(\cdot,\cdot)$ is $H$-invariant, $\alpha$ and $\beta$ are $H$-invariant. By definition, they are non-degenerate, bilinear and also symmetric and skew-symmetric, respectively.
\end{proof}

\begin{rem} One might want to look for Hopf--Hecke algebras constructed from completely reducible orthogonal $H$-modules $V$ with a non-degenerate $H$-invariant skew-symmetric bilinear form. Then \Cref{modules-with-two-forms} tells us that these modules are exactly the ones of the form $W\oplus W^*$.

Now if we take $H$ to be the universal enveloping algebra of the Lie algebra of a reductive algebraic group, a class of such Hopf--Hecke algebras called \emph{infinitesimal Cherednik algebras} is defined in \cite{EGG}.
\end{rem}

\begin{rem} The infinitesimal Hecke algebras of $\mathsf{Sp}_{2n}$ with the standard module $V=\CC^{2n}$ classified in \cite[Sec.~4.1.2]{EGG} and studied in \cite{Kh-symp,TK,DT,LT} are not Hopf--Hecke algebras, since the module does not have a non-degenerate invariant symmetric form; this follows from the above discussion, for instance, because we saw that a simple module cannot have a symmetric and a skew-symmetric non-degenerate invariant form at the same time. 

However, the infinitesimal Hecke algebras of the orthogonal groups $\mathsf{O}_n$ (\cite{EGG,Ts}) and the infinitesimal Cherednik algebras $\GL_n$ are Hopf--Hecke algebras. In the following, we will study the Dirac cohomology of the infinitesimal Cherednik algebras for the groups $\GL_n$, whose representation theory has been investigated in \cite{EGG,DT,LT,Ti} (see \Cref{rem-inf-cher-gln}). It seems a promising approach to use the Dirac operator to explore the somewhat less known representation theory of infinitesimal Hecke algebras of the orthogonal groups, which we hope to pursue in a future project.
\end{rem}

\subsection{Infinitesimal Cherednik algebras of \texorpdfstring{$\GL_n$}{GL\_n} as Hopf--Hecke algebras}
\newcommand{\vbar}{{\overline{v}}}
\newcommand{\dv}{{\,dv}}
\newcommand{\vvbar}{{v\tensor\vbar}}
\newcommand{\intv}{\int_{|v|=1}}

For a fixed $n\geq 1$ and with $\FF=\CC$ we consider the general linear group $G=\GL_n(\CC)$, its Lie algebra $\gfrak=\gl_n(\CC)$ and its universal enveloping algebra $H=\U(\gfrak)$. We consider the standard Lie algebra (and hence $H$-)module $\h=\CC^n$. We define the $H$-module $V:=\h\oplus \h^*$, where $\h^*$ is the usual contragredient module, and we denote the pairing of $\h^*$ and $\h$ by $(\cdot,\cdot)$.

The following definitions are from \cite{EGG}:
\begin{defn} \label{def-inf-hecke-algebra-gln} For all $m\geq 0$, $x\in\h^*$ and $y\in\h$, let $r_m(x,y)$ be the coefficient of $\tau^m$ in the expansion of the polynomial function $A\mapsto (x,(1-\tau A)\inv \cdot y)\det(1-\tau A)\inv$ in $S(\gl_n^*)$ viewed as an element in $S(\gl_n^*)\simeq S(\gl_n)\simeq\U(\gl_n)$, where the first identification is via the trace pairing $\gl_n\times\gl_n\to \CC, (A,B)\mapsto \Tr(AB)$ and the second identification is via the symmetrization map.

Let $\xi(z)=\sum_{m\geq0} \xi_m z^m$ be a polynomial. We define a map $\kappa=\kappa_\xi:V\wedge V\to H$ by
\begin{equation} \label{eq-kappa-series}
\kappa(x,x')=\kappa(y,y')=0
\ ,\quad
\kappa(y,x) := \sum_{m\geq 0} \xi_m r_m(x,y)
\tforall x,x'\in\h^*,y,y'\in\h
\ .
\end{equation}
Let $I_\kappa$ be the ideal of $T(V)\rtimes H$ generated by elements of the form $vw-wv-\kappa(v,w)$ for $v,w\in V$. The algebra
\[ \Hcal_\xi := (T(V)\rtimes H)/I_\kappa
\]
is called \emph{infinitesimal Cherednik algebra}.
\end{defn}

There is an alternative definition of $\kappa$ in terms of $\xi$ as explained in \cite[Sec.~4.2]{EGG} (see also \cite[Sec.~3.1]{DT}):
\begin{defn}\label{defn-f} 
Let $\tixi$ be the polynomial
\begin{equation}
 \tixi(z) := \frac{1}{2\pi^n} \partial^n(z^n \xi(z))
  = \sum_{m\geq0} \frac{1}{2\pi^n} \frac{(m+n)!}{m!} \xi_m z^m
\ .
\end{equation}
Also, we define the notations $\tri[v, w]_H:=v^T\overline{w}$, which is an Hermitian inner product on $\h$, and $|v|:=(\sum_i |v_i|^2)^{1/2}$ for all $v\in\h$, the Euclidean norm.

For every non-zero $v\in\h$, let $\vvbar$ denote the rank-one endomorphism $v\tri[\cdot,v]_H$ of $\h$ viewed as an element in $\gl_n$, so $\tixi(\vvbar)$ can be viewed as an element in $S(\gl_n)$ or $\U(\gl_n)$.
\end{defn}

We need the following lemma which is essentially contained in {\cite[Sec.~4.2]{EGG}}.

\begin{lem} With the definitions as above,
\begin{equation} \label{eq-kappa-int}
\kappa(y,x)=\intv (x,(\vvbar) \cdot y) \tixi(\vvbar) \dv
\tforall x\in\h^*, y\in\h
\ .
\end{equation}
\end{lem}

\begin{proof} We recall results from \cite[Sec.~4.2]{EGG}: Let $F_m\in S(\gfrak^*)$ be defined by
\[
 F_m(A):=\intv \tri[A \cdot v,v]_H^{m+1} \dv
 \tforall A\in\gl_n
 \ .
\]
According to the computations in \cite[Sec.~4.2]{EGG}, $F_m(A)$ equals the coefficient of $\tau^{m+1}$ in
\[ 2\pi^n \frac{(m+1)!}{(m+n)!} \det(1-\tau A)\inv
\ .
\]
As explained in \cite[Sec.~4.2]{EGG}, under the identification $S(\gfrak)\simeq S(\gfrak^*)$,
\begin{align*}
\intv (x,(\vvbar) \cdot y) (\vvbar)^m \dv
 &= \intv (x,(\vvbar) \cdot y) \tri[A \cdot v,v]_H^m \dv \\
 &= \frac{1}{m+1} dF_m|_A(y\tensor x)
 = 2\pi^n \frac{m!}{(m+n)!} r_m
\end{align*}
where $A\in\gfrak$ symbolizes the argument of a polynomial function in $S(\gfrak^*)$, and where $r_m$ is the coefficient of $\tau^m$ in $(x,(1-\tau A)\inv \cdot y)\det(1-\tau A)\inv$.

Now if we write $\tixi(z)=\sum_{m\geq 0} \tixi_m z^m$, then by definition, $\tixi_m = \frac{1}{2\pi^n} \frac{(m+n)!}{m!} \xi_m$ for all $m\geq 0$, so
\[
 \intv (x,(\vvbar) \cdot y) \tixi(\vvbar) \dv
 = \sum_{m\geq 0} 2\pi^n \frac{m!}{(m+n)!} \tixi_m r_m(x,y)
 = \sum_{m\geq 0} \xi_m r_m(x,y)
 \ .
\]
\end{proof}

\begin{rem}\label{rem-inf-cher-gln} In fact, \cite[Thm.~4.2]{EGG} says that $(T(V)\rtimes H)/I_\kappa$ has the PBW property if and only if $\kappa$ is of the described form (for some polynomial $\xi$). Since we will see that $V$ is an orthogonal $\U(\gl_n)$-module, the Hopf--Hecke algebras $A_{H,V,\kappa}$ with $H=\U(\gl_n)$ and $V=\h\oplus\h^*$ are exactly the infinitesimal Cherednik algebras.

We also note that the presentation of infinitesimal Cherednik algebras is in ``reverse order'' here: In \cite{EGG}, it is first explained that infinitesimal Cherednik algebras for a reductive algebraic group $G$ over $\CC$ are parametrized by $G$-invariant distributions on the closed subscheme of ``complex reflections'' $\Phi\subset G$ defined by $\wedge^2(1-g|_\h)=0$ which are supported at $1$. It is shown that for $G=\GL_n$, those distributions are parametrized by polynomials. The relation between the polynomials and the resulting deformations is computed to be \eqref{eq-kappa-int}. After evaluating the integral, the equivalent formulation \eqref{eq-kappa-series} is given.

The center of these algebras has been shown to be a polynomial algebra in $n$ variables in \cite{Ti}. Their representation theory has been studied and, in particular, their finite-dimensional irreducible modules have been classified in \cite{DT}. Universal infinitesimal Cherednik algebras, which are the analogs of infinitesimal Cherednik algebras with $\xi_1,\dots,\xi_n$ viewed as formal parameters, have been identified with $W$-algebras of the same type and a $1$-block nilpotent element in \cite{LT}.
\end{rem}

We want to see that $\Hcal_\xi$ is a Hopf--Hecke algebra in our notation and we want to find a description of $D^2$.

\begin{defn} \label{defn-tri-v} Let $(\cdot,\cdot):\h^*\tensor\h\to\CC$ be the natural pairing, which is $\gfrak$-invariant. We define a form $\tri$ on $V$ by
\[ \tri[x+y,x'+y']:=(x,y')+(x',y)
\tforall x,x'\in\h^*, y,y'\in\h
\ .
\]

We pick dual bases $(x_i)_i$, $(y_i)_i$ of $\h^*$ and $\h$, respectively, and we define
\[ (v_k)_k := (x_1,\dots,x_n,y_1,\dots,y_n),
\qquad
  (v^k)_k := (y_1,\dots,y_n,x_1,\dots,x_n)
\ .
\]
\end{defn}

\begin{lem} In the situation as in the definition, $\tri$ is a symmetric $\gfrak$-invariant bilinear form on $V$, i.e., $V$ is an orthogonal $H$-module and $\Hcal_\xi$ is a Hopf--Hecke algebra, and $(v_k)_k$, $(v^k)_k$ is a pair of dual bases for $V$ with respect to $\tri$.
\end{lem}

\begin{proof} $\tri$ makes $V$ an orthogonal $H$-module with the described pair of dual bases, because the natural pairing $(\cdot,\cdot)$ is $\gfrak$-invariant, as we have seen in \Cref{modules-with-two-forms} already.

By construction, $\Hcal_\xi$ has the PBW property, so it is a Hopf--Hecke algebra.
\end{proof}

Recall from the discussion in \Cref{sec-Dirac} that we can associate a pin cover to the Hopf--Hecke algebra $\Hcal_\xi$ which, as $H=\U(\gl_n)$, splits and is hence completely described by an algebra map $\gamma:H\to C$. To make this more concrete:

\begin{prop}\label{gamma-Eij} The pin cover $\tiH$ of $H$ splits, so $H'\cong H$ as algebras, we can identify $H$ with a Hopf subalgebra of $\tiH$ and $\gamma:\tiH\to C$, $\DeltaC:\tiH\to H\tensor C$ restrict to algebra maps $\gamma:H\to C$ and $\DeltaC=(\id_H\o\gamma)\circ\Delta:H\to H\tensor C$ (abusing notation). Furthermore, $\gamma(E_{ij})=\frac14(y_ix_j-x_j y_i)\in C$ for the elementary matrix $E_{ij}$ in $\gl(\h)\simeq \gl_n$ which sends $y_j$ to $y_i$, and the $\gl_n$-action on $C=C(V)$ via $\gamma$ coincides with the action induced from the action on the tensor algebra $T(V)$.
\end{prop}

\begin{proof} As discussed in \Cref{sec-Dirac}, the pin cover splits, because $H$ is the universal enveloping algebra of a Lie algebra. Consequently, $H$ can be identified with a Hopf subalgebra of $\tiH$, $H'\cong H$, and we have the restricted algebra maps as asserted.

To verify $\gamma(E_{ij})=\tfrac14(y_i x_j-x_j y_i)$, we realize that the action of $E_{ij}$ on $V=\h\oplus\h^*$ can be expressed as
\[ \tri[x_j,\cdot] y_i - \tri[y_i,\cdot] x_j
\ ,
\]
a skew-symmetric operator on $V$. Now $\gamma(E_{ij})$ is given as the image of this operator under $\phi\inv$ in $\Biv\subset C$ (see \Cref{rem-bivectors}), which is just $\tfrac14(y_i x_j - x_j y_i)$, as desired.

We can verify that the $\gl_n$-action on $C$ which we obtain through this algebra map coincides with the $\gl_n$-action which is induced by the $\gl_n$-action on the tensor algebra. 
\end{proof}

We recall the definitions of the Casimir element $\Omega=\sum_k v_k v^k$ in $A=\Hcal_\xi$ and of the Dirac element $D=\sum_k v_k\tensor v^k$ in $A\tensor C$ (\Cref{defn-Casimir-Dirac}) for any pair of dual bases $(v_k)_k$ and $(v^k)_k$, so, in particular, for the choice made in \Cref{defn-tri-v}.

\begin{lem} Let $D\in A\tensor C$ be the Dirac element for $A=\Hcal_\xi$. Then
\begin{equation} \label{formula-D2-integral}
 D^2 = \Omega\tensor 1
  - 2\intv \tixi(\vvbar)\tensor \gamma(\vvbar) \dv
\ .
\end{equation}
\end{lem}

\begin{proof} We invoke \Cref{square-of-D} to obtain
\begin{align*}
D^2
 &= \Omega\tensor1
 + \tfrac12 \sum_{k<l} \kappa(v_k,v_l) \tensor [v^k,v^l]
 = \Omega\tensor1
 + \tfrac12 \sum_{i,j} \kappa(y_j,x_i) \tensor [x_j,y_i] \\
 &= \Omega\tensor1
 - 2 \sum_{i,j} \kappa(y_j,x_i) \tensor \gamma(E_{ij})
\ ,
\end{align*}
where $E_{ij}=y_i\tensor x_j$ as above is an element in $\gl_n$ for all $i,j$. Using the integral formula \eqref{eq-kappa-int} for $\kappa$, we obtain
\begin{align*}
 \sum_{i,j} \kappa(y_j,x_i)\tensor\gamma(E_{ij})
 &= \sum_{i,j} \intv \tixi(\vvbar)\tensor (x_i,(\vvbar)y_j)\gamma(E_{ij}) \dv \\
 &= \intv \tixi(\vvbar)\tensor \gamma(\vvbar) \dv
\ ,
\end{align*}
as desired.
\end{proof}

In the following, we want to find an even more explicit expression for $D^2$ in terms of polynomials derived from $\tixi$, which will allow us to prove that $D$ satisfies the Parthasarathy condition and hence $\Hcal_\xi$ is a Barbasch--Sahi algebra. We need some auxiliary lemmas.

From now on, all polynomials are univariate with complex coefficients unless otherwise stated.

\begin{defn}\label{def-nabla-eps} For any $\eps\in\CC$, we define $\nabla_\eps$, a difference operator on polynomials, by
\[ \nabla_\eps f(z) := f(z+\eps)-f(z+\eps-1)
\ .
\]
For $k\geq 0$, let $B_k(z)$ be the $k$-th \emph{Bernoulli polynomial} defined by the generating series
\[ \sum_{k\geq0} B_k(z) \frac{t^k}{k!} = \frac{te^{tz}}{e^t-1}
\]
(\cite[Eq.~23.1.1]{AbS}). We recall that $B_k$ satisfies $\nabla_1 B_k(z) = B_k(z+1)-B_k(z) = k z^{k-1}$ (\cite[Eq.~23.1.6]{AbS}).

\end{defn}

\begin{lem} \label{lem-Bernoulli} Let $p$ be a polynomial and $\eps\in\CC$. Then there is a polynomial $f$ satisfying $\nabla_\eps f(z)=p(z)$ and $f$ is characterized by this relation uniquely up to the constant term.
\end{lem}

\begin{proof}
To construct $f$, we write $p(z)=\sum_{i\geq0} p_i z^i$. Then
\[
 \nabla_\eps f(z)=p(z)
 \quad\Leftrightarrow\quad
 \nabla_1 f(z)=p(z+1-\eps)=\sum_{i\geq0} \frac{p_{i}}{i+1} (i+1) (z+1-\eps)^{i}
\ ,
\]
hence $f(z):=\sum_{i\geq0} \frac{p_i}{i+1} B_{i+1}(z+1-\eps)+f_0$ satisfies this recurrence relation for any scalar $f_0$.

For uniqueness, let $g$ be another polynomial satisfying the same recurrence relation. Then $f_d=f-g$ is a polynomial satisfying $\nabla_\eps f_d(z)=0$. Hence $f_d$ attains the same value at, say, all integers, so it has to be a constant polynomial.
\end{proof}

\begin{lem} \label{lem-polynomials} For a fixed polynomial $p$, let $f$ be a polynomial satisfying $\nabla_{1/2} f(z)=p(z)$. Then
\begin{equation}
p(z)\omega = f(z+\omega) + \tfrac12 p(z) - f(z+\tfrac12)
\qquad
\mbox{in }\CC[z,\omega]\mod(\omega^2-\tfrac14)
  \cdot
\end{equation}
\end{lem}

\begin{proof} We claim that for every polynomial $p$, there are polynomials $f,q$ such that
\begin{equation} \label{eq-lemma-pq}
 p(z)\omega=f(z+\omega)+q(z)
\qquad
\mbox{in }\CC[z,\omega]\mod(\omega^2-\tfrac14)
\ .
\end{equation}
First we note that it is enough to show this for polynomials $p$ of the form $p(z)=(k+1)z^k$, because those form a basis. Consider $k=0$. Then $p(z)\omega=\omega=(z+\omega)-z$, which verifies the claim. Assume the claim is true for all non-negative integers $0\leq k<K$ for some $K\geq 1$, and hence for all polynomials $p$ of degree at most $K-1$. We consider $p(z)=(K+1)z^K$ and $f(z)=z^{K+1}$, then
\[ p(z) \omega = f(z+\omega) + p'(z)\omega + q(z)
\]
for polynomials $p',q$ with $\deg p'\leq K-1$, because $\omega^2\equiv\tfrac14$ and the coefficients of $z^K \omega$ equal $(K+1)$ on both sides. This proves the claim by induction.

We assume now $f,q$ are as in \eqref{eq-lemma-pq}. Then we can substitute $\omega=\pm\tfrac12$ to get
\begin{equation}\label{eq-q-p-f}
q(z) = \pm\tfrac12 p(z) - f(z\pm\tfrac12)
\ .
\end{equation}
However, the two choices of substitution should yield the same result, so
\[ \tfrac12 p(z)-f(z+\tfrac12)=-\tfrac12 p(z)-f(z-\tfrac12)
\qquad\Leftrightarrow\qquad
f(z+\tfrac12)-f(z-\tfrac12)=p(z)
\ .
\]
Choosing the positive sign in \eqref{eq-q-p-f} together with \eqref{eq-lemma-pq} yields
\[ p(z)\omega
 = f(z+\omega) + \tfrac12 p(z)-f(z+\tfrac12)
 \ ,
\]
as desired.
\end{proof}

\begin{lem} \label{lem-gamma-square} Let $v$ be a vector in $\h$ with $|v|=1$, and let $\vvbar$ be the corresponding rank-one matrix in $\gl_n$. Then $\gamma(\vvbar)^2=\tfrac14$ in $C(V)$.
\end{lem}

\begin{proof} We write $v=\sum_i v_i y_i$, then by linearity of $\gamma$,
\[ \gamma(\vvbar)
 =\sum_{i,j} v_i \vbar_j \gamma(E_{ij})
 =\tfrac14 \sum_{i,j} v_i \vbar_j [y_i,x_j]
 =\tfrac14 \left[\sum_i v_i y_i,\sum_i \vbar_i x_i\right]
 = \tfrac14[v,v^*]
\ ,
\]
where $v$ and $v^*:=\sum_i \vbar_i x_i$ can be regarded as elements of $V$ or of $C(V)$, and where we used the value of $\gamma(E_{ij})$ as discussed in \Cref{gamma-Eij}.

Now in $C(V)$, $v^2=\tri[v,v]=0$, $(v^*)^2=\tri[v^*,v^*]=0$ and
$vv^*+v^*v= 2 \tri[v,v^*]= 2$. Hence,
\begin{align*}
\gamma(\vvbar)^2
 &= \tfrac{1}{16}(vv^*vv^* + v^*vv^*v-v(v^*)^2v-v^*v^2v^*)
 = \tfrac{1}{16}(v(2-vv^*)v^* + v^*(2-v^*v)v) \\
 &= \tfrac{1}{8}(vv^*+v^*v)
 = \tfrac14
 \ ,
\end{align*}
as desired.
\end{proof}

We are ready to give a refined formula for $D^2$.

\begin{defn}
\label{defn-polynomial-p}\label{def-nu}
Let $f_\xi(z)$ be the polynomial uniquely defined by $f_\xi(0)=0$ and
\[
\nabla_0 f_\xi(z) = f_\xi(z) - f_\xi(z-1)
 = \tixi(z)
 = \tfrac{1}{2\pi^n} \partial^n(z^n \xi(z))
\]
(the first and the last equality being the definitions of $\nabla_0$ and $\tixi$, respectively).
Furthermore, we define $\alpha,\beta\in\U(\gl_n)$ by
\[ \alpha := \intv  -\tixi(\vvbar) + 2f_\xi(\vvbar) \dv
\ ,
\qquad
 \beta := \intv 2f_\xi(\vvbar-\tfrac12) \dv
\ ,
\]
and $t'_1=\tfrac12(\Omega+\alpha)$.
\end{defn}

\begin{rem} Let us compare this with objects studied in \cite{DT}: The polynomial $f_\xi$ corresponds to the polynomial called ``$2\pi^n f$'' there and the element $t'_1$ is the Casimir element studied and denoted by the same symbol in the reference. In \cite{Ti}, it is proved that the center of $\Hcal_\xi$ is freely generated by a total of $n$ (``higher Casimir'') elements (see \Cref{rem-center-Hxi} below). 
\end{rem}

In \cite{DT} it is in particular shown that $t'_1$ is central in $\Hcal_\xi$. We include a slightly different argument for this statement when proving the following formula of $D^2$ and the Parthasarathy condition for $A$.

\begin{prop} \label{D^2} \label{Hxi-Vogan}
Let $f:=f_\xi$, $\alpha,\beta$ as in the definition. Then we have the following formula for $D^2$:
\begin{equation} \label{eq-D^2}
D^2 = (\Omega+\alpha)\tensor1 - \DeltaC(\beta) = 2t'_1\tensor1 - \DeltaC(\beta)
\ .
\end{equation}

Furthermore, $t'_1$ is central in $\Hcal_\xi$ and $\beta$ is central in $H$. In particular, $D$ satisfies the Parthasarathy condition and $\Hcal_\xi$ is a Barbasch--Sahi algebra.
\end{prop}

\begin{proof} We fix $v\in\h$ with $|v|=1$ and define elements $z:=(\vvbar)\tensor1$, $\omega:=1\tensor\gamma(\vvbar)$ in $A\tensor C$. Then $z+\omega=\DeltaC(\vvbar)$ and $\omega^2=\tfrac14$ by \Cref{lem-gamma-square}. We observe that $\nabla_{1/2} f(z-\tfrac12)=\nabla_0 f(z)=\tixi(z)$ by the definition of $f=f_\xi$. So we can apply \Cref{lem-polynomials} to obtain
\[ \tixi(\vvbar)\tensor\gamma(\vvbar)
 = f(\DeltaC(\vvbar)-\tfrac12)
  + (\tfrac12 \tixi(\vvbar) - f(\vvbar))\tensor 1
\ ,
\]
which yields the new formula for $D^2$ when substituted into \eqref{formula-D2-integral}.

We define the shorthand $M_v:=\vvbar\in\gl_n$ for any $v\in\h$ to show now that $\Omega+\alpha$ is central in $A$. First we note that
\begin{align*}
\Omega+\alpha
 &= \sum_i (x_i y_i + y_i x_i)+\alpha
 = \sum_i (2x_iy_i + [y_i,x_i])+\alpha
 \\
 &= \sum_i (2x_iy_i + \intv (x_i,M_v \cdot y_i) \tixi(M_v)\dv)+\alpha
 = 2\sum_i x_iy_i + 2\intv f(M_v)\dv
\ ,
\end{align*}
where we use that $\sum_i (x_i,(\vvbar) \cdot y_i) = \sum_i |v_i|^2 = 1$.
Also, when verifying centrality, it suffices to consider a set of algebra generators, say, the elements of $\h,\h^*$ and $\gl_n$, respectively.

So let us fix $y,v\in\h$ such that $|v|=1$ and $M:=M_v$. We regard $M$ as an element in a universal enveloping algebra, so $M^k$ denotes a tensor power of $M$ for all $k\geq 0$. If $\mu:\gl_n^{\tensor k}\to \gl_n$ is the matrix multiplication, we have $\mu(M^k)=M$ for all $k\geq 1$, so we can compute in $A=T(V)\rtimes\U(\gl_n)$:
\[ M^k y
 = \sum_{i=0}^k \binom{k}{i} (\mu(M^{k-i}) \cdot y)M^i
 = yM^k+\sum_{i=0}^{k-1} \binom{k}{i} (M \cdot y)M^i
 = yM^k+(M \cdot y)(M+1)^k-(M \cdot y)M^k
 \ ,
\]
because $M$ is a primitive element, so the coproduct of $M^k$ is just $\sum_{i=0}^k \binom{k}{i} M^{k-i}\tensor M^i$. Hence,
\[ [M^k,y]=(M \cdot y)((M+1)^k-M^k)
\]
for all $k\geq 0$ and hence for any polynomial $q$,
\[ [q(M),y]=(M \cdot y) \nabla_1 q(M)
\ .
\]
In particular,
\[
 \left[\intv f(M_v)\dv,y\right]
  = \intv (M_v \cdot y) \nabla_1 f(M_v)\dv
  = \intv (M_v \cdot y) \tixi(M_v+1) \dv
  \ .
\]
On the other hand,
\begin{align*}
\left[\sum_i x_i y_i,y\right]
 &= \sum_i [x_i,y] y_i
 = -\intv \sum_i (x_i,M_v \cdot y) \tixi(M_v) y_i \dv \\
 &= -\intv [\tixi(M_v),M_v \cdot y] + (M_v \cdot y)\tixi(M_v) \dv
 = -\intv (M_v \cdot y) \tixi(M_v+1) \dv
 \ ,
\end{align*}
where we have used that $M_v \cdot (M_v \cdot y)=M_v \cdot y$. So indeed, $\Omega+\alpha$ commutes with any $y\in\h$. A parallel argument shows that $\Omega+\alpha$ commutes with any $x\in\h^*$. (Alternatively, this follows from the existence of an anti-involution of $\Hcal_\xi$ sending $y_i\leftrightarrow x_i$ and $E_{ij}\leftrightarrow E_{ji}$ as described in \cite[Sec.~2]{DT}.)

Furthermore, we have seen already that $\Omega$ commutes with elements from $\gl_n$, so it remains to show that $\alpha$ and $\beta$ are central in $\U(\gl_n)$, too. Let $q$ be any polynomial and consider the element $h_q=\intv q(M_v)\dv$ in $\U(\gl_n)$. We note that $h_q$ is invariant under the adjoint action of $U(\h)\subset\GL(\h)$, the unitary group of $\h$ with respect to $\tri_H$, because $QM_vQ^*=M_{Qv}$ for all $Q\in U(\h), v\in\h$ and the integral is invariant under the transformation $v\mapsto Qv$. Now $\gl_n$ is just the complexified Lie algebra of $U(\h)$, so the center of $\U(\gl_n)$ is just the space of $U(\h)$-invariants in $\U(\gl_n)$. Hence $h_q$ is central in $\U(\gl_n)$, and in particular, $\alpha$ and $\beta$ are central in $H=\U(\gl_n)$.

Now $D$ satisfies the Parthasarathy condition, because $(\Omega+\alpha)\tensor 1$ is central in $\AC$ and $\beta$ is in $H=H\ieven$ (see \Cref{sec-Dirac}).
\end{proof}

\begin{cor}\label{cor-zeta} There is an algebra map $\zeta\colon Z(\Hcal_\xi)\to Z(\U(\gl_n))$ relating the central characters and Dirac cohomology for $\Hcal_\xi$-modules, in the following sense: For any $\Hcal_\xi$-module $M$ with non-zero Dirac cohomology $H^D(M)$ and with a central character, the central character equals $\phi\circ\zeta$ for any (non-zero) irreducible $\gl_n$-submodule $(\phi,N)$ of $H^D(M)$.
\end{cor}

\begin{proof} This is \cite[Thm.~4.3]{Fl} in the special case considered here.
\end{proof}


\subsection{Dirac cohomology for \texorpdfstring{$\Hcal_\xi$}{H\_xi}} Having seen that $\Hcal_\xi$ is what we call a Barbasch--Sahi algebra, we can explore the Dirac cohomology of its modules. Here we will focus on finite-dimensional irreducible modules, which have been studied in \cite{DT}.

In the following, we identify $\l=(\l_1,\dots,\l_n)\in\CC^n$ with the $\gl_n$-weight $\l_1 E_{11}^*+\dots+\l_n E_{nn}^*$. We denote the set of dominant integral $\gl_n$-weights by $\L^+$, that is, $\l\in\L^+$ if and only if $(\l_i-\l_{i+1})$ is a non-negative integer for all $1\leq i<n$. For $\l\in\L^+$, let $V_\l$ be the finite-dimensional irreducible highest weight $\gl_n$-module with highest weight $\l$.

Let $S$ be the spin module of the Clifford algebra $C=C(V)$ (which is unique, since $\dim V$ is even). Then $S$ is a $\gl_n$-module via the algebra map $\gamma:\U(\gl_n)\to C$ which is defined by $\gamma(E_{ij})=\tfrac14(y_ix_j-x_jy_i)$ for all elementary matrices $E_{ij}\in\gl_n$ (see \Cref{gamma-Eij}).  We have the following information on the structure of $S$ as a $\gl_n$-module via $\gamma$ (see \cite[Prop.~3.17]{Ko-multiplets}):

\begin{lem} \label{lem-weights-S} The weights of $S$ are exactly the weights $(s_1,\dots,s_n)$ in $\{\pm\tfrac12\}^n$, and all weight spaces are one-dimensional. Hence $S\cong \Lambda(\h)\tensor(-\tfrac12\Tr)$ as $\gl_n$-modules.
\end{lem}

\begin{proof} We can take $S$ to be the left ideal generated by $u:=y_1\dots y_n$ in $C(V)$, which is irreducible (this is explained, for instance, in \cite[Sec.~3]{Ko-multiplets}). Hence, a basis of $S$ is given by the elements $x_1^{e_1}\dots x_n^{e_n} u$ for exponents $e_1,\dots,e_n\in \{0,1\}$.  We can compute directly
\begin{align*}
\gamma(E_{ii}) x_i
 &= \tfrac14(y_i x_i - x_i y_i) x_i
 = -\tfrac14 x_i y_i x_i
 = -\tfrac12 x_i
 \ ,
 \\
\gamma(E_{ii}) y_i
 &= \tfrac14(y_i x_i - x_i y_i) y_i
 = \tfrac14 y_i x_i y_i
 = \tfrac12 y_i
 \ ,
\end{align*}
and $\gamma(E_{ii})$ commutes with $x_j$ or $y_j$ in $C(V)$ for all $j\neq i$, so
\[ \gamma(E_{ii}) x_1^{e_1}\dots x_n^{e_n} u
= \tfrac12 (-1)^{e_i} x_1^{e_1}\dots x_n^{e_n} u
\]
for all $1\leq i\leq n$. Similarly, we find that $E_{i,i+1} x_1\dots x_j u=0$  for all $1\leq i<n$ and $1\leq j\leq n$, i.e., $x_1\dots x_j u$ is a highest weight vector, which yields the desired description of $S$ as a $\gl_n$-module.
\end{proof}

From here we can go on to compute the action of $D^2$ and the Dirac cohomology for all finite-dimensional irreducible $\Hcal_\xi$-modules. These were classified in \cite{DT} and we now recall the classification.

\begin{defn}[$M(\l),L(\l)$] \label{def-verma} For any $\gl_n$-weight $\l$, let $M(\l)$ be the \emph{Verma module} of $\Hcal_\xi$ defined by 
\[ M(\l) = \Hcal_\xi/( \Hcal_\xi E_{ij} +\Hcal_\xi y_k + \Hcal_\xi(E_{kk}-\l_k))_{i<j,k}
\ ,
\]
where $E_{ij}\in\gl_n$ are the elementary matrices as before, and let $L(\l)$ be the unique irreducible quotient (which exists similar to the classical case).
\end{defn}

The following says that the Dirac cohomology determines $\l$ both for Verma modules and their irreducible quotients.

\begin{lem} \label{highest-HDM-weight} There is an irreducible $\gl_n$-submodule with highest weight $\l+(\tfrac12,\dots,\tfrac12)$ in the Dirac cohomology of both $M(\l)$ and $L(\l)$, and $\l+(\tfrac12,\dots,\tfrac12)$ is the highest $\gl_n$-weight occurring. 
\end{lem}

\begin{proof} \newcommand\nfrak{{\mathfrak{n}}}
Let $M$ be $M(\l)$ or $L(\l)$, let $m$ be the image of $1\in\Hcal_\xi$ in the factor space $M$ and define $u:=y_1\dots y_n\in S$ as above. Then $y_i m=y_i u=0$ for all $i$. Hence, the Dirac operator $D=\sum_i x_i\o y_i+y_i\o x_i$ acts as $0$ on $m\o u$ in $M\o S$.

On the other hand, we observe that since $M(\l)$ is isomorphic to (a quotient of) $S(\nfrak^-)\o S(\h^*)$ as a $\gl_n$-module, where $\nfrak^-$ is the span of $\{E_{ij}\}_{i>j}$, both $M$ and $S$ are direct sums of their $\gl_n$-weight spaces and both have a unique maximal weight $\l$ and $(\tfrac12,\dots,\tfrac12)$, respectively. But $x_i$ lowers the $\gl_n$-weight for each $i$, so the vector $m\o u$ of weight $\l+(\tfrac12,\dots,\tfrac12)$ cannot be in the image of the Dirac operator. 

Hence, the image of $m\o u$ in $H^D(M)$ generates a (non-zero) irreducible $\gl_n$-submodule of highest weight $\l+(\tfrac12,\dots,\tfrac12)$.
\end{proof}

\newcommand\C{\mathcal{C}}

\begin{defn}[$h_k,T_a,\nabla,w^p,\rho,\C(p,\mu)$]\label{def-P} ~
\begin{itemize}
\item For $k\geq 0$, let $h_k=h_k(z_1,\dots,z_n)$ be the \emph{complete homogeneous symmetric polynomial $h_k$} of degree $k$ in the variables $z_1,\dots,z_n$, that is, 
\[ h_k(z_1,\dots,z_n)
:= \sum_{l_1+\dots+l_n=k,l_i\geq0} z_1^{l_1}\dots z_n^{l_n}
\ .
\]
\item For $a\in\CC$, let $T_a$ be the translation operator for polynomials, i.e., $T_a p(z):=p(z+a)$ for any polynomial $p$. Let $\nabla:=\nabla_{1/2}$ (see \Cref{def-nabla-eps}), that is, $\nabla=T_{1/2}-T_{-1/2}$.
\item For any polynomial $p$, let $w^p$ be the polynomial uniquely defined by
\[ \nabla^{n-1} z^{n-1} w(z) = 2\pi^n p(z)
\ .
\]
\item Let $\rho:=(\tfrac{n-1}2,\tfrac{n-1}2-1,\dots,-\tfrac{n-1}2)\in\CC^n$ be the Weyl vector of $\gl_n$.
\item For any polynomial $p$ and any dominant integral $\gl_n$-weight $\mu$, let $\C(p,\mu)$ denote the scalar by which the central element $\intv p(v\o\vbar)\dv$ of $\U(\gl_n)$ (see \Cref{D^2} and its proof) acts on $V_\mu$, the finite-dimensional irreducible $\gl_n$-module with highest weight $\mu$. 
\end{itemize}
\end{defn}

\begin{prop} For any polynomial $p$, any $\mu\in\L^+$ and any $a\in\CC$,
\begin{equation}\label{eq-assertion-C}
\C(p,\mu) = \sum_{k\geq0} w^p_k h_k(\mu+\rho)
\quad\text{and}\quad 
\C(T_a p,\mu) = \C(p,\mu+(a,\dots,a))
\ .
\end{equation}
\end{prop}

\begin{proof} The first identity is proven in \cite[Sec.~3.2]{DT} using the Weyl character formula by taking a suitable limit. To derive the second identity from the first, let us note that
\[
\nabla^{n-1} z^{n-1} w^{T_a p}(z) 
= T_a(2\pi^n p(z))
= T_a \nabla^{n-1} z^{n-1} w^p(z)
\ .
\] 
Now $T_a$ commutes with $\nabla^{n-1}$, which annihilates polynomials of degree at most $n-2$. Hence, for any $k\geq 0$,
\begin{equation}\label{eq-proof-C-1}
 T_a \nabla^{n-1} z^{n-1} z^k
 = \nabla^{n-1} (z+a)^{k+n-1}
 = \nabla^{n-1} z^{n-1} \sum_{0\leq i\leq k} \binom{k+n-1}{i} a^i z^{k-i}
 \ .
\end{equation} 

We claim that similarly,
\begin{equation}\label{eq-proof-C-2}
 h_k(z_1+a,\dots,z_n+a)
 = \sum_{0\leq i\leq k} \binom{k+n-1}{i} a^i h_{k-i}(z_1,\dots,z_n)
\ .
\end{equation} 
If $n=1$, then the claim is clearly true. Assume it is true for $n-1$, then the expression on the left-hand side can be simplified to be
\begin{multline*}
\sum_{0\leq i\leq k} h_i(z_1+a,\dots,z_{n-1}+a) (z_n+a)^{k-i}
\\
= \sum_{i_1+i_2+i_3+i_4=k} \binom{i_1+i_2+n-2}{i_1} a^{i_1} h_{i_2}(z_1,\dots,z_{n-1}) \binom{i_3+i_4}{i_4} z_n^{i_3} a^{i_4}
\\
= \sum_{i_2+i_3+i=k}
\binom{k+n-1}{i}  a^i
h_{i_2}(z_1,\dots,z_{n-1}) z_n^{i_3}
= \sum_{0\leq i\leq k} \binom{k+n-1}{i} a^i h_{k-i}(z_1,\dots,z_n)
\ ,
\end{multline*}
where all summation indices are non-negative and where we used the identity 
\[ \sum_{i_1+i_4=i} \binom{n_1+i_1}{i_1}\binom{n_2+i_4}{i_4}=\binom{n_1+n_2+i+1}{i}
\]
for arbitrary integers $n_1,n_2$, a special case of the Rothe--Hagen identity (see \cite[Eq.~(3)]{Go}). This proves the claim by induction.

The identities \Cref{eq-proof-C-1} and \Cref{eq-proof-C-2} directly imply that for all weights $\mu$,
\[
 \sum_{k\geq0} w^{T_a p}_k h_k(\mu)
 = \sum_{k\geq0} \sum_{0\leq i\leq k} \binom{k+n-1}{i} a^i w^p_k h_{k-i}(\mu)
 = \sum_{k\geq0} w^p_k h_k(\mu_1+a,\dots,\mu_n+a)
 \ ,
\]
which is equivalent to the second identity in \Cref{eq-assertion-C}.
\end{proof}


\begin{defn}[$w,P$]
Let $w=w(z):=w^{f_\xi(z)}$ and let $P$ be the multivariate polynomial 
\[ P(\mu):=\C(f_\xi,\mu)=\sum_{m\geq 0} w_k h_k(\mu+\rho)
\ .
\]
\end{defn}
Note that eventually, $w$ and $P$ depend only on the deformation parameter $\xi$ of $\Hcal_\xi$.

\begin{defn}[$\tiLambda,\nu$] Furthermore, we define the set
\[ \tiLambda
 :=\{\l\in\L^+: \exists k\in\ZZ_{\geq0}: P(\l)=P(\l-(0,\dots,0,k+1))\}
 \ ,
\]
and for any $\l\in\tiLambda$, we define $\nu=\nu(\xi,\l)\in\ZZ_{\geq0}^n$ by letting $\nu_i$ be the minimal non-negative integer such that $\l':=\l-(0,\dots,0,\nu_i+1,0,\dots,0)$ is either not a dominant weight or $P(\l)=P(\l')$ for every $1\leq i\leq n$.
\end{defn}

The set $\tiLambda$ parametrizes the finite-dimensional irreducible $\Hcal_\xi$-modules, each of which can be thought of as a rectangular grid of irreducible $\gl_n$-modules:

\begin{prop}{{\cite[Thm.~3.2, Thm.~4.1]{DT}}} \label{dt-fd-modules} The finite-dimensional irreducible modules of $\Hcal_\xi$ are given (up to isomorphism) by the set $\{L(\l)\}_{\l\in\tiLambda}$. For each $\l\in\tiLambda$, the central (Casimir) element $t'_1$ acts on $L(\l)$ by the scalar $P(\l)$ and 
\[
 L(\l)=\bigoplus_{0\leq \nu'\leq \nu} V_{\l-\nu'}
\]
as $\gl_n$-modules, where $\nu=\nu(\xi,\l)$ is as defined above and $\nu'\in\ZZ_{\geq0}^n$ runs over all tuples satisfying $0\leq \nu'_i\leq \nu_i$ for all $1\leq i\leq n$.
\end{prop}

We can use this result to obtain the structure of $L(\l)\tensor S$: let us fix the deformation parameter $\xi$, the highest weight $\l\in\tiLambda$ and $\nu=\nu(\xi,\l)$ as above. 
\begin{defn}[$\theta_k(a), m_\mu$]\label{def-decomposition-LS} For any dominant integral $\gl_n$-weight $\mu$, we define 
\[ m_\mu:=\prod_{1\leq i\leq n} \theta_{\nu_i}(\l_i+\tfrac12-\mu_i)
\ ,\quad\textnormal{where }
\theta_k(a):=\begin{cases} 
0 & a\not\in\{0,\dots,k+1\} \\
1 & a\in\{0,k+1\} \\
2 & a\in\{1,\dots,k\}
\end{cases}
\quad\textnormal{for } a\in\CC,k\in\ZZ_{\geq0}
\ .
\]
\end{defn}

\begin{prop}\label{decomposition-LS} For each $\l\in\tiLambda$,
\[ L(\l)\tensor S = \bigoplus_\mu m_\mu V_\mu
\]
as $\gl_n$-modules, where the sum ranges over all dominant integral $\gl_n$-weights $\mu$; in particular, the weights occurring with non-zero $m_\mu$ are those satisfying
\[ \mu_i\in\{\l_i+\tfrac12,\l_i-\tfrac12,\dots,\l_i-\nu_i-\tfrac12\}
\tforall 1\leq i\leq n
\ .
\]
\end{prop}

\begin{proof} Let $\l'$ be a dominant integral $\gl_n$-weight and $V_{\l'}$ the corresponding irreducible highest weight $\gl_n$-module. Then by the Pieri rule, $V_{\l'}\tensor\Lambda(\h)$ decomposes as
\[
 \bigoplus \{ V_\mu : \mu\in\L^+,
 	\mu_i-\l'_i\in\{0,1\}
 	\ \forall 1\leq i\leq n
 \}
 \ .
\]
Now since $L(\l)=\bigoplus_{0\leq\nu'\leq\nu} V_{\l-\nu'}$ and $S=\Lambda(\h)\tensor(-\tfrac12\Tr)$, $L(\l)\tensor S$ decomposes as 
\begin{equation} \label{eq-decomposition-LS}
 \bigoplus_{0\leq \nu'\leq \nu}
 \bigoplus \{ V_\mu : \mu\in\L^+,
 	 \mu_i-(\l_i-\nu'_i)\in\{\pm\tfrac12\}
 	 \ \forall 1\leq i\leq n
 	\}
 \ .
\end{equation}
To determine the multiplicity of a weight $\mu$, we count the number of ways we can write $\mu_i$ as $\l_i-\nu'_i\pm\tfrac12$ for some $0\leq\nu'_i\leq\nu_i$, for each $i$. This number is just $\theta_{\nu_i}(\l_i+\tfrac12-\mu_i)$.
\end{proof}

\begin{prop} \label{kernel-D2} For each $\l\in\tiLambda$, the kernel of $D^2$ acting on $L(\l)\tensor S$ is the $\gl_n$-module 
\[ \bigoplus \{m_\mu V_\mu :  P(\l)
  =P(\mu-(\tfrac12,\dots,\tfrac12))\}
\ .
\]
\end{prop}

\begin{proof} We recall that 
\[ 
D^2 = 2 (\sum_i x_i y_i + \intv f_\xi(\vvbar)\dv)\o 1 - 2\DeltaC(\intv f_\xi(\vvbar-\tfrac12)\dv )
\]
according to \Cref{D^2}. The sum of the two terms in parentheses is just the central Casimir element $t'_1$ in $\Hcal_\xi$ which acts on $L(\l)$ by the scalar $\C(f_\xi,\l)=P(\l)$, because all $y_i$ act as $0$ on the irreducible $\gl_n$-submodule of $L(\l)$ with highest weight $\l$. Let us consider an irreducible $\gl_n$-submodule of $L(\l)\o S$ with highest weight $\mu$. Then $\tfrac12 D^2$ acts on it by the scalar
\[  P(\l) - \C(T_{-1/2} f_\xi, \mu)
 = P(\l) - \C(f_\xi, \mu - (\tfrac12,\dots,\tfrac12))
  = P(\l) - P(\mu - (\tfrac12,\dots,\tfrac12))
 \ .
\]
\end{proof}


\renewcommand{\ll}[1]{\lambda^{(#1)}}
\begin{defn}[$\ll{I}, \ll{i}, I^+$] For any $I\subset\{1,\dots,n\}$, we define $\ll{I}\in\CC^n$ by
\[ \ll{I}_j := \l_j + \tfrac12 - \begin{cases} \nu_j + 1 & j\in I \\ 0 & j\not\in I \end{cases}
\quad\text{for all } 1\leq j\leq n
\]
and for any $1\leq i\leq n$, $\ll{i}:=\ll{\{i\}}$.
\end{defn}

Note that $\ll{\varnothing}=\l+(\tfrac12,\dots,\tfrac12)$ and that according to \Cref{def-decomposition-LS} and \Cref{decomposition-LS}, the set $\L^+\cap\{\ll{I}\}_I$ consists exactly of those weights $\mu$, for which $m_\mu=1$, that is, for which $L(\l)\o S$ has a unique irreducible $\gl_n$-submodule of the respective highest weight. Let us denote the unique submodules by $V^{(I)}$ whenever $\ll{I}$ is dominant. These modules can be thought of as the extremal vertices of the rectangular grid formed by all $\gl_n$-submodules of $L(\l)\o S$.

\begin{lem}\label{lem-V0-Vi} $V^{(\varnothing)}$ and $V^{(i)}$ are contained in the kernel of $D^2$ whenever $\ll{i}$ is dominant.
\end{lem}

\begin{proof} Clearly
\[ P(\ll{\varnothing}-(\tfrac12,\dots,\tfrac12)) = P(\l)
\]
and 
\[ P(\ll{i}-(\tfrac12,\dots,\tfrac12) 
 = P(\l_1,\dots,\l_i - \nu_i - 1,\dots,\l_n)
 = P(\l)
 \ ,
\]
by the definition of $\nu$ (\Cref{def-nu}) if $\ll{i}$ is dominant.
\end{proof}

To say more, we need the following observation:

\begin{lem}\label{lem-h-rectangle} Let $h(z_1,z_2)$ denote a linear combination of complete homogeneous symmetric polynomials in two variables. Consider $d_1,d_2\geq0$, $a_1,a_2\in\CC$ satisfying $a_1-a_2\in\RR_{>0}$,
\[ r:= h(a_1,a_2) = h(a_1,a_2-d_2)\ ,\]
and one of the following conditions:
\begin{itemize}
    \item $a_1-d_1-a_2 = 0$
    \item $a_1-d_1-a_2\in\RR_{>0}$ and $h(a_1-d_1,a_2) = r$ .
\end{itemize}
Then $h(a_1-d_1,a_2-d_2) = r$.
\end{lem}

\begin{proof} First we note that we can express the complete homogeneous symmetric polynomial $h_k(z_1,z_2)$ of any degree $k\geq0$ as the quotient $(z_1^{k+1}-z_2^{k+1})/(z_1-z_2)$ as long as $z_1\neq z_2$. Hence, if we write $h=\sum_{k\geq0} s_k h_k$ with coefficients $(s_k)_k$ and if we define the polynomial $p(z):=\sum_{k\geq0} s_k z^{k+1}$, then $h(z_1,z_2) = (p(z_1)-p(z_2))/(z_1-z_2)$ as long as $z_1\neq z_2$.

Next, let us record that
\begin{equation}\label{eq-fractions}
    r = s_1/t_1 = s_2/t_2 \Rightarrow r = (s_1+s_2) / (t_1+t_2)
\end{equation}
for all $s_1,s_2,t_1,t_2\in\CC$ such that $t_1,t_2,$ and their sum are non-zero.

Now if $d_1=0$ or $d_2=0$, the statement is trivial, so may assume $d_1$ and $d_2$ are positive. Also we may assume $a_2<a_1=0$ for the proof, because the general statement follows from this special case using an argument shift and \Cref{eq-proof-C-2}. 

Thus we have
\[ r
 = p(a_2)/a_2
 = p(a_2-d_2)/(a_2-d_2)
 \ .
\] 
and $-d_1-a_2=0$ or $-d_1-a_2>0$ together with $(p(-d_1)-p(a_2))/(-d_1-a_2)=r$. Both imply
\[
r = p(-d_1)/(-d_1)
\ ,
\]
either directly or with \Cref{eq-fractions}. Another application of \Cref{eq-fractions} now yields
\[ r = (p(-d_1) - p(a_2-d_2)) / (-d_1 - (a_2-d_2)) = h(-d_1,a_2-d_2)
\ ,
\]
as desired.
\end{proof}

We have already observed that the set $\{V^{(I)}\}_I$ parametrizes the irreducible $\gl_n$-submodules of $L\o S$ which are located at the vertices of the rectangular grid formed by all its irreducible $\gl_n$-submodules. The following results mean that these vertices, and hence, the shape of the rectangular grid, are captured by the Dirac cohomology.

\begin{lem}\label{lem-multiplicity-one} Let $M$ be an $\Hcal_\xi$-module and assume $N$ is an irreducible $\gl_n$-submodule which appears with odd multiplicity (e.g., multiplicity one) in $\ker D^2\subset M\o S$. Then $N$ appears in $H^D(M)$.
\end{lem}

\begin{proof} The action of $D$ is a $\gl_n$-module map by \Cref{lem-D-Delta-commute}, so $D$ acts on the multiplicity space of $N$ such that its square is zero. If the multiplicity space has an odd dimension, this action has a non-trivial cohomology. 
\end{proof}

\begin{prop}\label{V^I-in-H^D} $V^{(I)}$ is contained in the Dirac cohomology $H^D(L(\l))$ for each $I\subset\{1,\dots,n\}$ if $\ll{I}$ is dominant.
\end{prop}

\begin{proof} By \Cref{lem-multiplicity-one} it is enough to show that all $V^{(I)}$ are contained in the kernel of $D^2$, because these $\gl_n$-submodules have multiplicity one in $L(\l)\o S$.

We have seen in \Cref{lem-V0-Vi} that $V^{(I)}$ lies in the kernel of $D^2$ if $|I|\leq 1$, and we can proceed by induction in $|I|$. Assume $l:=|I|\geq 2$, $\ll{I}\in\L^+$ and assume all $V^{(I')}$ with $|I'|<l$ and $\ll{I'}\in\L^+$ lie in the kernel of $D^2$.

Pick $i<j$, the two smallest indices in $I$. Then as $\l$ and $\ll{I}$ are dominant integral, $\ll{I\setminus\{i\}}$ and $\ll{I\setminus\{i,j\}}$ are dominant and $\ll{I\setminus\{j\}}$ is dominant unless $j=i+1$ and (again due to the minimality condition on $\nu$)
\[ \lambda_i - \nu_i - 1 = \lambda_{i+1} - 1
\ .
\]

This means that
\[ P(\l) 
= P(\ll{I\setminus\{i\}}-(\tfrac12,\dots,\tfrac12))
= P(\ll{I\setminus\{i,j\}}-(\tfrac12,\dots,\tfrac12))
=: r
\ .
\]
Moreover, if $\ll{I\setminus\{j\}}$ is dominant, then 
\[
 P(\ll{I\setminus\{j\}}-(\tfrac12,\dots,\tfrac12)) = r
\ .
\]

Let us define $a_1:=\l_i+\rho_i$, $a_2:=\l_j+\rho_j$, $d_1:=\nu_i+1$, $d_2:=\nu_j+1$, and
\[
h(z_1,z_2):=\sum_k w_k h_k(\l_1+\rho_1,\dots,z_1,\dots,z_2,\dots,\l_n+\rho_n)
\ ,
\]
where $z_1$ and $z_2$ are the $i$-th and $j$-th argument, respectively. Then the properties of $P$ we have discussed above mean that $r=h(a_1,a_2)=h(a_1,a_2-d_2)$ and $h(a_1-d_1,a_2)=r$ if $\ll{I\setminus\{j\}}$ is dominant, or else $a_1-d_1=a_2$.

This means we can apply \Cref{lem-h-rectangle} to conclude that $h(a_1-d_1,a_2-d_2)=r$, that is, $P(\l)=P(\ll{I}-(\tfrac12,\dots,\tfrac12))$, i.e., $V^{(I)}$ lies in the kernel of $D^2$, which completes the induction step.
\end{proof}

\begin{cor}\label{cor-Dirac-determines-L} Any finite-dimensional irreducible representation of the infinitesimal Cherednik algebra $H_\xi$ is uniquely determined by its Dirac cohomology. The highest weight $\l$ and the dimensions $\nu$ of the rectangle formed by the highest $\gl_n$-weights can be read off from the Dirac cohomology.
\end{cor}

\begin{proof} $\ll{\varnothing}$ is the maximal weight occurring in the Dirac cohomology and $\l=\ll{\varnothing}-(\tfrac12,\dots,\tfrac12)$. 

For each $1\leq i\leq n$, either there is an irreducible $\gl_n$-submodule in the Dirac cohomology whose highest weight is of the form $\ll{\varnothing}-(0,\dots,k,0,\dots,0)$, where $k>0$ is the $i$-th coordinate, in which case $\nu_i=k-1$, or there is no such submodule, in which case $\nu_i$ is maximal among those non-negative integers $k$ for which $\l-(0,\dots,k,0,\dots,0)$ is dominant.
\end{proof}

Let us conclude our discussion with examples for $n=1$ and $n=2$ (cf.~the examples in \cite[Sec.~4]{DT}).

\begin{expl} 
For $n=1$, $\l$ is a complex number and $\nu$ is a non-negative integer (minimal) such that $P(\l)=P(\l-\nu-1)$. Then $L(\l)=V_\l \oplus\dots\oplus V_{\l-\nu}$. Hence,
\[
L(\l)\tensor S
 = V_{\l+\tfrac12} \oplus 2V_{\l-\tfrac12} \oplus \dots \oplus 2V_{\l-\nu+\tfrac12} \oplus V_{\l-\nu-\tfrac12}
 \ .
\]
Now the only weights $\mu$ occurring in $L(\l)\tensor S$ such that $P(\l)=P(\mu-\tfrac12)$ are obviously $\ll{\varnothing}=\l+\tfrac12$ and $\ll{1}=\l-\nu-\tfrac12$. So the kernel of $D^2$ and, by \Cref{V^I-in-H^D}, the Dirac cohomology is just $V^{(\varnothing)}\oplus V^{(1)}$.
\end{expl}

\begin{expl}\label{expl-n2}
For $n=2$, we identify weights with points in the plane and we consider the $\gl_n$-weight $\l=(2.5,0.5)$ together with a polynomial $P$ in two variables (whose existence we will establish shortly) satisfying
\[
P(\l) = P(\l-(0,3))
\text{ and }
P(\l-(k,0)) \neq P(\l)\neq P(\l-(0,k))
\text{ for }k=1,2
\ .
\]
Then $\nu_1=2$, because $\l-(3,0)=(-0.5,0.5)$ is not dominant and $\nu_2=2$, because $P(\l-(0,3))=P(\l)$ by assumption. Hence the irreducible $\gl_n$-submodules occurring in $L(\l)$ form a $3\times 3$-grid and their highest weights are those $\mu$ satisfying $\l\geq\mu\geq\l-(2,2)$. Each of these irreducible $\gl_n$-submodules produces up to $4$ irreducible $\gl_n$-submodules when tensored with the spin module $S$, for an irreducible $\gl_n$-module with highest weight $\mu$ they have highest $\gl_n$-weights $\mu+(\pm\tfrac12,\pm\tfrac12)$ (see \Cref{fig-n2}).

\begin{figure}
\centering
\raisebox{-.5\height}{\includegraphics[width=0.5\textwidth]{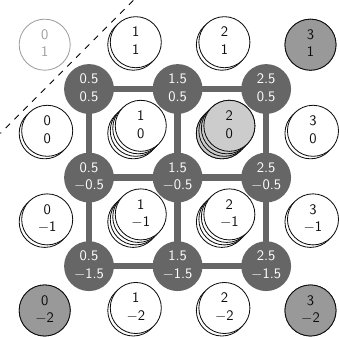}}
\qquad\qquad
\begin{tabular}{c||cccc|}
     $\nu'_2$ \textbackslash $\nu'_1$ & 3 & 2 & 1 & 0 \\   
     \hline\hline
     0 & 0 & 13 & 19 & \cellcolor{gray!30} 9 \\
     1 & -11 & 1 & \cellcolor{gray!30}9 & 4 \\
     2 & -11 & -3 & 4 & 1 \\
     3 & \cellcolor{gray!30}9 & 10 & 13 & \cellcolor{gray!30}9 \\ 
     \hline
\end{tabular}
\caption{\textit{left:} Weights of a finite-dimensional module $M$ (filled dark circles), of the tensor product $M\tensor S$ (all circles with dark outlines indicating multiplicities), and of the kernel of $D^2$ (shaded circles with dark outlines indicating multiplicities). According to \Cref{V^I-in-H^D}, the three multiplicity-free weights at the vertices in $\ker D^2$ appear in the Dirac cohomology, and as we discuss in \Cref{expl-n2} they, in fact, form the full Dirac cohomology (so the slightly lighter shaded circles labeled $(2,0)$ are in the kernel of $D^2$, but not in the Dirac cohomology). The weight in the top left corner is not dominant. \newline\textit{right:}  The values of the polynomial $P(\mu)=\tfrac{27}2 h_1(\mu+\rho)+ h_2(\mu+\rho)-\tfrac32 h_3(\mu+\rho)$ for $\mu=\l-\nu'$.
}\label{fig-n2}
\end{figure}

Now we specialize $P(\mu)=\tfrac{27}2 h_1(\mu+\rho)+h_2(\mu+\rho)-\tfrac32 h_3(\mu+\rho))$ with the complete homogeneous symmetric polynomials in two variables $h_1,h_2,h_3$. We can check (see the table in \Cref{fig-n2}) that $P$ satisfies the conditions mentioned so far and also
\[ P(\l)=P(\l-(3,3))=P(\l-(1,1)) \ .
\]
Note that $P(\l)=P(\l-(3,3))$ follows already from our previous assumptions by \Cref{lem-h-rectangle}.

Hence, the kernel of $D^2$ is the sum of the irreducible $\gl_n$-submodules of $L(\l)\o S$ with highest weights $(3,1)$, $(3,-2)$, $(2,0)$, or $(0,-2)$ with their multiplicities $1$ or $4$, respectively.

By \Cref{V^I-in-H^D}, the three modules with multiplicity one also occur in the Dirac cohomology. Let us determine the contribution of the remaining weight $(2,0)$ with multiplicity $4$ in $L(\l)\o S$ to the Dirac cohomology.

We view $L(\l)$ as a factor space of $\Hcal_\xi$ identifying elements of the latter space with their images under the quotient map and $S$ as the left ideal of $C$ generated by $u=y_1\dots y_n$, as before. It can then be verified that
\begin{multline*}
m_1 := 1\o x_1 x_2 u \ ,\quad m_2 := x_1\o x_2 u-x_2\o x_1 u\ ,\quad
\\
m_3 := x_2 (x_2 E_{21} + (\l_1-\l_2)x_1 ) \o u \ ,\quad
m_4 := (x_2 E_{21} + (\l_1-\l_2)x_1 ) \o x_2 u
\end{multline*}
are four linearly independent highest weight vectors of irreducible $\gl_n$-submodules with highest weight $(2,0)$ in $L(\l)\o S$. Moreover,
\begin{gather*} 
D m_1 = (x_1\o y_1+x_2\o y_2) m_1 = 2 m_2  \\
D m_4 = (y_1\o x_1+x_2\o y_2) m_4 = 2 m_3 + r m_1
\end{gather*}
for some $r\in\CC$. This means that $D$ has rank at least $2$ on the multiplicity space of $V_{(2,0)}$ with respect to $L(\l)\o S$, but as $D$ squares to $0$ on this space, the rank is exactly $2$, the action of $D$ decomposes into two Jordan blocks of eigenvalue $0$ and size $2$ each, and $D$ has no cohomology on this space.

Hence, the Dirac cohomology for this module equals the sum of those irreducible submodules in the kernel of $D^2$ which are multiplicity-free, which is the part of the Dirac cohomology described by \Cref{V^I-in-H^D}.
\end{expl}

The part of the Dirac cohomology described by \Cref{V^I-in-H^D} is already the full Dirac cohomology for all examples we computed.

\bigskip

Finally, let us conclude with a description of the map $\zeta$ from the center of $\Hcal_\xi$ to the center of $\U(\gl_n)$ which relates central characters according to \Cref{cor-zeta}.

Let $\beta_1,\dots,\beta_n$ be the standard generators of the center of $\U(\gl_n)$ which can be obtained as the coefficients of $\tau$ in the series expansion of the polynomial function $A\mapsto \det(A-\tau)$, using a suitable identification $S(\gl_n^*)\simeq\U(\gl_n)$, as explained in \cite[Sec.~2]{Ti}. It is shown in \cite{Ti} that the center of $\Hcal_\xi$ is freely generated by elements
\begin{equation} \label{rem-center-Hxi}
\eta_i := \sum_{1\leq k\leq n} [\beta_i,y_k]x_k - c_i
\end{equation}
for elements $c_i=c_i(\xi)\in Z(\U(\gl_n))$ for all $1\leq i\leq n$.

\begin{defn} Using Sweedler's notation $\Delta(u)=u_{(1)}\o u_{(2)}$ for elements $u\in\U(\gl_n)$, we define for $1\leq i\leq n$:
$$
b_i := \sum_k \beta_i[y_k,x_k] - [y_k,(\beta_i)_{(1)}\cdot x_k] (\beta_i)_{(2)} - c_i \in \U(\gl_n)
\ .
$$
\end{defn}

\begin{lem} $b_i$ is central in $\U(\gl_n)$.
\end{lem}

\begin{proof} We use that $\beta_i$ and $c_i$ are central in $\U(\gl_n)$ and $\sum_k y_kx_k$ is an $\gl_n$-invariant element, so for any $u\in\U(\gl_n)$,
\begin{align*}
u b_i
 &= \sum_k \beta_i[u_{(1)}\cdot y_k,u_{(2)}\cdot x_k] u_{(3)} - [u_{(1)}\cdot y_k,(u_{(2)}(\beta_i)_{(1)})\cdot x_k] u_{(3)} (\beta_i)_{(2)} S(u_{(4)}) u_{(5)}  - u c_i
 \\
 &= \sum_k \beta_i[y_k,x_k] u - [y_k,(\beta_i)_{(1)}\cdot x_k] (\beta_i)_{(2)} u  - c_i u
 = b_i u
 \ ,
\end{align*}
as desired.
\end{proof}

\newcommand\HC{\mathsf{HC}}
\newcommand\rhovec{(\tfrac12,\dots,\tfrac12)}
\newcommand\cT{\mathcal{T}}
Let $\HC\colon Z(\U(\gl_n))\to \CC[\lambda_1,\dots,\lambda_n]^{S_n}$ be the Harish-Chandra isomorphism between the center of $\U(\gl_n)$ and the symmetric polynomial functions in $n$ variables.
So for any $z\in Z(\U(\gl_n))$ and any integral dominant $\gl_n$-weight $\lambda$, $\HC(z)$ is a symmetric polynomial function in $n$ variables and $\HC(z)(\lambda)$ is its evaluation at $\lambda$, which is the scalar by which $z$ acts on $V_{\lambda-\rhovec}$, the finite-dimensional irreducible $\gl_n$-module with highest weight $\lambda-\rhovec$. 

We observe that for any symmetric polynomial function $h(\lambda_1,\dots,\lambda_n)$, there is another symmetric polynomial function $h(\lambda_1-\tfrac12,\dots,\lambda_n-\tfrac12)$. Let us denote by $\cT$ the map which is induced by this translation on $Z(\U(\gl_n))$ via $\HC$.

\begin{prop} \label{prop-explicit} Then $\zeta\colon Z(\Hcal_\xi)\to Z(\U(\gl_n))$ is the unique algebra map sending $\eta_i\mapsto \cT(b_i)$.
\end{prop}

\begin{proof} For all $u\in \U(\gl_n)$, $x\in\h^*$, $y\in\h$, we can compute
\begin{align*}
[u,y]x 
&= u y x - y u x 
= (u[y,x] + u x y ) - y (u_{(1)}\cdot x) u_{(2)}
\\
&= (u[y,x] + u x y) - [y,u_{(1)}\cdot x] u_{(2)} - (u_{(1)}\cdot x)y u_{(2)}
\equiv u[y,x] - [y,u_{(1)}\cdot x] u_{(2)}
\end{align*}
modulo $\Hcal_\xi \h$, the left ideal in $\Hcal_\xi$ generated by $\h$, where we use that $\h\U(\gl_n)=\U(\gl_n)\h$ in $\Hcal_\xi$. So 
$$ \eta_i\equiv b_i \qquad\text{ modulo } \Hcal_\xi \h
\ .
$$
Let us recall that for any deformation parameter $\xi$ and any dominant integral $\gl_n$-weight $\lambda$, there is a Verma module $M(\lambda)$ of $\Hcal_\xi$ (see \Cref{def-verma}).
Now $\eta_i$ acts on $M(\lambda)$ by $\HC(b_i)(\lambda+\rhovec)$, since $\lambda$ is the highest weight corresponding to an irreducible $\gl_n$-module in $M(\lambda)$, and any element from $\h$ acts as $0$ on it. 
On the other hand, we know from \Cref{highest-HDM-weight} that there is an irreducible $\gl_n$-submodule with highest weight $(\lambda+\rhovec)$ in $H^D(M(\l))$, on which $\zeta(\eta_i)$ has to act by the same scalar, so
$$
\HC(\zeta(\eta_i))(\lambda+2\rhovec) = \HC(b_i)(\lambda+\rhovec)
$$
which implies the asserted relation between $\eta_i$ and its image under $\zeta$.
\end{proof}



\bigskip

\end{document}